\documentclass[11pt]{article}
\usepackage{caption}
\usepackage{graphicx}
\usepackage{subcaption}
\usepackage{bbm}
\usepackage{amsmath,amsthm,amssymb}
\usepackage{pdflscape}
\usepackage{booktabs}
\usepackage{graphicx} 
\usepackage{tikz}
\usepackage{chemformula}

\usepackage{fontawesome5}
\usepackage{adjustbox}
\usepackage[utf8]{inputenc}
\usepackage{pgfplots}
\pgfplotsset{compat=1.17} 
\usepackage[colorlinks=true,urlcolor=blue,linkcolor=blue,citecolor=blue]{hyperref}
\usepackage{setspace}
\usepackage{pict2e,color}
\usepackage{multirow}
\usepackage{booktabs}
\usepackage{amsfonts}
\usepackage{verbatim} 
\usepackage{sectsty}
\usepackage{url}
\usepackage{empheq}
\usepackage[normalem]{ulem}
\usepackage[titletoc,title]{appendix}
\usepackage[flushleft]{threeparttable}
\usepackage{soul} 
\usetikzlibrary{shapes.geometric}
\usepackage{algorithm}
\usepackage{algpseudocode}
\usepackage{tikz}
\usepackage{mathtools}
\usepackage{pbox}
\usepackage{etoolbox}
\preto\subequations{\ifhmode\unskip\fi}
\usepackage{booktabs}
\usepackage{multirow}
\usepackage{xcolor}
\usepackage[numbers]{natbib}
\usepackage{natbib}
\theoremstyle{definition}

\newtheorem{theorem}{Theorem}

\newtheorem{proposition}{Proposition}

\newtheorem{assumption}{Assumption}

\allowdisplaybreaks
\usetikzlibrary{positioning, calc}
\tikzstyle{startstop} = [rectangle, rounded corners, minimum width=3cm, minimum height=1cm,text centered, draw=black, fill=red!30]
\tikzstyle{process} = [rectangle, minimum width=3cm, minimum height=1cm, text centered, draw=black, fill=orange!30]
\tikzstyle{arrow} = [thick,->,>=stealth]

\usepackage{threeparttable}

\usepackage{natbib}
\setcitestyle{authoryear} 

\title{Contextual Stochastic Optimization with Decision-Dependent Uncertainty via Nonparametric Learning}

\author{Huangrong Sun, Xian Yu\thanks{Corresponding author}\\
\footnotesize{The Ohio State University, Columbus, OH (USA)}\\
\footnotesize{\url{sun.3394@osu.edu}, \url{yu.3610@osu.edu}}
}
\date{}
\pgfplotsset{compat=1.18}

\begin{document}

\maketitle
\vspace*{-0.1in}
\begin{abstract}
    We study a general decision-dependent contextual stochastic program (DD-CSP) in which uncertainty depends on both exogenous contextual information and endogenous decisions. To learn the potentially complex dependence of uncertainty on decisions and contextual information, we employ several nonparametric regression models, including \(k\) nearest neighbors (\(k\)NN), classification and regression trees (CART), and ReLU neural networks. To account for estimation errors in predicting the uncertainty, we adopt an empirical residuals-based decision-dependent sample average approximation (ER-DD-SAA) framework, which adds empirical residuals to the point predictions from the learned regression models. 
For each nonparametric regression model, we develop exact mixed-integer programming (MIP) representations that can be seamlessly embedded within the ER-DD-SAA framework. Since \(k\)NN leads to the most computationally challenging MIP model and has been less studied in the literature, we derive both a pairwise distance comparison formulation and an equivalent but more compact bilevel formulation. For two-stage ER-DD-SAA problems with $k$NN, we further propose a tailored decomposition algorithm, named BD-CG, that combines Bender's decomposition
with constraint generation. 
Under suitable assumptions, we prove that the proposed BD-CG converges to a global optimum within a finite number of iterations. From a statistical perspective, we establish the consistency and asymptotic optimality of ER-DD-SAA with all three nonparametric regression models under mild regularity conditions. Numerical experiments on a newsvendor problem with pricing and a two-stage facility location problem demonstrate that the ER-DD-SAA model with nonparametric learning consistently outperforms a parametric benchmark in out-of-sample performance and the proposed reformulations and algorithm substantially improve computational tractability.

\end{abstract}



\section{Introduction}
We consider the following decision-dependent contextual stochastic program (DD-CSP):
\begin{align}\label{eq:dd-csp}
    \psi^\ast(\boldsymbol{w}) := 
    \min_{\boldsymbol{z} \in \mathcal{Z}} \ 
    \mathbb{E}_{\boldsymbol{Y}}\left[
    c(\boldsymbol{z},  \boldsymbol Y) 
    \mid \boldsymbol{Z}=\boldsymbol{z},\boldsymbol{W}=\boldsymbol{w}
    \right],
\end{align}
where \(\boldsymbol{z} \in \mathcal{Z}\subseteq \mathbb{R}^{d_z}\) denotes the decision vector, \( \boldsymbol{w}\in \mathcal{W} \subseteq \mathbb{R}^{d_w}\) denotes the contextual information, also known as \textit{covariates}, \textit{features}, or \textit{side information}, \(\boldsymbol{Y} \in \mathcal{Y} \subseteq \mathbb{R}^{d_y}\) represents the uncertainty that depends on both decision \(\boldsymbol{z}\) and 
covariate \(\boldsymbol{w}\), and $c(\boldsymbol{z}, \boldsymbol{Y})$ denotes the cost function. Given an observation of the covariate $\boldsymbol{W}=\boldsymbol{w}$, the expectation in \eqref{eq:dd-csp} is taken with respect to the conditional distribution of $\boldsymbol{Y}$ conditioned on the decision $\boldsymbol{z}$ to be optimized and the observed covariate $\boldsymbol{w}$.

Such DD-CSPs arise naturally in many real-world applications, where uncertainty is influenced by both exogenous contextual information and endogenous decisions, a phenomenon known as \textit{decision-dependent uncertainty} (DDU). For example, in a newsvendor problem with pricing, the decision-maker (DM) chooses the selling price and order quantity to minimize the expected cost under uncertain demand. In practice, the pricing decision directly affects customer demand, where a higher selling price may suppress demand, and a lower selling price may stimulate purchases \citep{liu2022coupled,bertsimas2020predictive}. At the same time, exogenous covariates, such as customer income, seasonal conditions, and local market characteristics, may also affect demand. Therefore, the uncertain demand depends jointly on the pricing decision and the contextual information. As another example, consider a two-stage facility location problem. In the first stage, the DM decides which facilities to open from a set of candidate locations. In the second stage, after customer demand is realized, shipments from the open facilities are allocated to customers to satisfy demand while minimizing the recourse cost. Facility location decisions may influence demand through service accessibility: opening a closer facility makes the service more convenient for customers to access, thereby increasing customer demand \citep{basciftci2021distributionally, yu2022multistage, mahmutougullari2023robust, li2024adaptive}. Meanwhile, contextual factors, such as population density, income level, and other regional characteristics, also affect demand \citep{sun2026contextual}. Therefore, customer demand depends jointly on the facility location decisions and the contextual information.

With the growing availability of data, uncertainty can be estimated from historical observations, albeit subject to prediction error. To account for both prediction errors and the aforementioned DDU, we adopt an empirical residuals-based decision-dependent sample average approximation (ER-DD-SAA) framework to approximate the DD-CSP. Specifically, we first train a regression model using historical data and compute the corresponding empirical residuals. Given a new realization of the contextual information, we then construct decision-dependent scenarios by adding these empirical residuals to the point prediction generated by the learned regression model, and solve the resulting SAA problem. Since the relationship between uncertainty, decisions, and non-decision covariates can be highly nonlinear and complex, we employ nonparametric learning methods, including $k$ nearest neighbors ($k$NN), classification and regression trees (CART), and rectified linear unit (ReLU) neural networks (NNs). Unlike the decision-independent setting, where uncertainty scenarios can be estimated solely from covariates prior to solving the optimization problem, the decision variables now become part of the input to the nonparametric regression models. Consequently, the learned regression models must be embedded directly into the downstream optimization problem. Although these nonparametric regression models can be trained efficiently, representing the learned regressors within the ER-DD-SAA framework and solving the resulting optimization problem require new mathematical formulations and efficient solution algorithms.


To address these challenges, we develop exact mixed-integer programming (MIP) formulations for ER-DD-SAA with the three nonparametric regressors described above. To this end, we assume that the feasible region $\mathcal{Z}$ is linear programming (LP)- or mixed-integer linear programming (MILP)-representable. We further assume that the projection onto $\mathcal{Y}$ admits an MILP representation, e.g., when the support set $\mathcal{Y}$ is the real space $\mathbb{R}^{d_y}$, a half-space, or a hyperrectangle. Under these assumptions, Table~\ref{tab:method-summary} summarizes the resulting reformulation types---either MILP or mixed-integer nonlinear programming (MINLP)---for both objective uncertainty and right-hand-side (RHS) uncertainty.

\begin{table}[htbp]
\centering
\caption{Formulation types under different nonparametric regressors with objective and RHS uncertainty.}
\label{tab:method-summary}
\begin{tabular}{lcc}
\hline
Regressor
& Objective Uncertainty\textsuperscript{a}
& RHS Uncertainty\textsuperscript{b}
\\ \hline
$k$NN    & MILP  & MILP \\
CART     & MILP  & MILP \\
ReLU NNs & MINLP & MILP \\
\hline
\end{tabular}

\vspace{2pt}
\begin{tabular}{@{}l@{\ }l@{}}
\textsuperscript{a} & Single-stage or two-stage stochastic (MI)LP with objective uncertainty.\\
\textsuperscript{b} & Two-stage stochastic (MI)LP with RHS uncertainty.
\end{tabular}
\end{table}

\subsection{Related Work}

In recent years, several paradigms have emerged in contextual stochastic optimization to effectively incorporate predictive information from data into downstream optimization problems. These include ``smart predict-then-optimize'' \citep{elmachtoub2022smart, el2019generalization, estes2023smart}, estimate-then-optimize \citep{bertsimas2020predictive, bertsimas2023dynamic}, feature-to-decision learning \citep{ban2019big, bertsimas2022data, zhang2024optimal, qi2024learning}, and empirical residuals-based approaches \citep{ban2019dynamic, liu2022coupled, kannan2020residuals,kannan2022data}. We refer readers to \cite{qi2022integrating} and \cite{sadana2024survey} for comprehensive reviews of this area. 
While these existing works primarily consider decision-independent uncertainty, in many real-world applications, decisions themselves can also influence the uncertainty distribution, giving rise to DDU. Motivated by this observation, a growing body of work has incorporated DDU into a variety of applications, including facility location \citep{basciftci2021distributionally, yu2022multistage, liu2022coupled, mahmutougullari2023robust}, newsvendor problems with pricing \citep{liu2023solving, bertsimas2020predictive}, and appointment scheduling \citep{homem2022simulation}. Although these models capture endogenous uncertainty, they typically do not account for exogenous contextual information.
To bridge this gap, recent work has begun to study DD-CSP, in which uncertainty depends jointly on decisions and contextual information. \citet{bertsimas2020predictive} proposed a reweighted SAA framework in which scenario weights are determined jointly by endogenous decisions and exogenous covariates, and solved the resulting problem by enumerating all candidate decisions. \citet{cao2024statistical} studied DD-CSP under both predict-then-optimize and estimate-then-optimize paradigms, allowing for arbitrary dependence structures and establishing non-asymptotic guarantees on approximation error and decision regret. More recently, \citet{sun2026contextual} extended the empirical residuals-based SAA (ER-SAA) framework proposed in \cite{kannan2022data} to the decision-dependent setting, namely, ER-DD-SAA, for an electric vehicle charging station location problem. Specifically, the authors assumed that customer demand depends jointly on facility location and capacity decisions as well as contextual information. They estimated demand using both parametric and nonparametric regression models and incorporated the learned predictors into an ER-DD-SAA framework. However, their paper focuses on a specific application, and their one-step nonparametric learning relies on Gurobi Machine Learning \citep{gurobiMLfeatures}, which is limited to predictors supported by the package (e.g., $k$NN is not supported) and does not provide explicit MIP formulations. Our paper addresses these limitations by considering a general DD-CSP and deriving exact MIP formulations for ER-DD-SAA with $k$NN, CART, and ReLU NNs. We further establish statistical guarantees for ER-DD-SAA under these three nonparametric regression models.

We next review the nonparametric regression models considered in this paper. Nonparametric regression models are widely used for their ability to capture complex, nonlinear relationships without imposing a prespecified functional form \citep{hardle1990applied}. Among them, the $k$NN method predicts the response at a query point by averaging the responses of its nearest neighbors \citep{altman1992introduction}.
CART recursively partition the feature space into axis-aligned regions to approximate the underlying response function by piecewise constant functions \citep{breiman2017classification}. NNs approximate nonlinear functions by composing affine transformations with nonlinear activation functions \citep{goodfellow2016deep}. Among these activation functions, ReLU is widely used in modern NN architectures \citep{nair2010rectified}. Embedding these trained regressors into downstream optimization problems requires exact mathematical optimization formulations. A growing body of research has developed MIP formulations for trained ML models. For tree-based models, much of the literature has focused on formulating the decision tree training problem as a mixed-integer optimization problem \citep{bertsimas2017optimal, aghaei2025strong}. For ReLU NNs, numerous studies have derived strong MIP formulations for the ReLU activation function \citep{tong2024optimization, tong2025optimization, pham2025optimization, badilla2023computational, anderson2020strong}. In contrast, exact MIP formulations for trained $k$NN regressors have received comparatively little attention.

At the intersection of CSP, DDU, and nonparametric regression, only a few studies have considered this setting.  \citet{liu2023solving} studied decision-dependent newsvendor problems and incorporated nonparametric models such as $k$NN and kernel regression into a reweighted SAA framework by developing exact reformulations together with an approximate gradient descent algorithm for scalability. However, their approach is limited to continuous decision variables, as the gradient-based algorithm is not applicable to discrete decisions. More recently, \cite{liu2026solving} studied contextual chance-constrained programs under
DDU based on a reweighted SAA framework combined with cluster-based
nonparametric regression models (e.g., $k$NN and CART). Their solution approach, however, requires the decisions that affect the
uncertainty to take values in a finite discrete set so that cluster
memberships can be pre-computed. In contrast, our approach accommodates both continuous and discrete decision variables and is built upon the ER-DD-SAA framework, which naturally supports a broad class of nonparametric regression models. 

\subsection{Summary of main contributions}
Our key contributions are summarized as follows:
\begin{enumerate}
    \item We study a general DD-CSP problem in which the uncertainty depends jointly on decisions and covariates. We propose a unified ER-DD-SAA framework that can integrate a broad range of nonparametric regression models into downstream stochastic optimization.
    \item We develop exact MIP reformulations for ER-DD-SAA with $k$NN, CART, and ReLU NNs under appropriate structural conditions. For ER-DD-SAA with \(k\)NN, we propose a pairwise distance comparison formulation and an equivalent but more compact bilevel formulation.
    \item For two-stage ER-DD-SAA with \(k\)NN, we further develop a tailored decomposition algorithm, named BD-CG, that combines Bender's decomposition with constraint generation, and establish its finite convergence to global optimality under suitable conditions.
    \item We establish statistical guarantees for the proposed ER-DD-SAA framework with the three nonparametric regression models by proving consistency and asymptotic optimality under mild conditions.
    \item We conduct numerical experiments on a newsvendor problem with pricing and a two-stage facility location problem to evaluate the computational performance and solution quality of the proposed methods.
\end{enumerate}

The remainder of this paper is organized as follows. Section~\ref{sec:preliminaries} reviews the ER-DD-SAA framework and the three nonparametric regression models considered in this paper. Section~\ref{sec:MIP-formulations} develops exact MIP formulations for ER-DD-SAA with $k$NN, CART, and ReLU NNs. Section~\ref{sec:knn_tractability} presents a tailored decomposition algorithm for two-stage ER-DD-SAA with \(k\)NN to enhance computational tractability. Section~\ref{sec:consistency} establishes statistical guarantees for the ER-DD-SAA framework under mild conditions. Finally, Section~\ref{sec:experiments} reports numerical results on a newsvendor problem with pricing and a two-stage facility location problem, demonstrating the effectiveness of the proposed approaches.

\section{Preliminaries}\label{sec:preliminaries}
This section presents the ER-DD-SAA framework and introduces the three nonparametric regression models considered in this paper.

\subsection{ER-DD-SAA}
We assume that the uncertainty \(\boldsymbol{Y}\) has the following relationship: 
\(
    \boldsymbol{Y} = Q^\ast(\boldsymbol{z}, \boldsymbol{w}) + \boldsymbol{\epsilon} \in  \mathcal{Y},
\)
where $Q^\ast(\cdot, \cdot)$ is the true (unknown) regression function, and $\epsilon$ is the zero-mean additive error term. Under this assumption, the original DD-CSP problem \eqref{eq:dd-csp} is equivalent to
\begin{align}\label{eq:er-dd-saa-true}
    \psi^\ast(\boldsymbol{w})
    := \min_{\boldsymbol{z} \in \mathcal{Z}}
    \left\{
      \upsilon(\boldsymbol{z}, \boldsymbol{w})
    := \mathbb{E}_{\boldsymbol{\epsilon}}\Big[c\Big(\boldsymbol{z},  Q^\ast(\boldsymbol{z}, \boldsymbol{w}) + \boldsymbol{\epsilon} \Big) \Big]
    \right\}.
\end{align}
Here, we assume that $\mathcal{Z}$ is a nonempty, compact set that is LP- or MILP-representable, the objective function $v(\cdot,\boldsymbol{w})$ is lower semicontinuous on $\mathcal{Z}$ for every $ \boldsymbol{w} \in\mathcal{W}$, and $\mathbb{E}_{\boldsymbol{\epsilon}}\left[|c\left(\boldsymbol{z},  Q^\ast(\boldsymbol{z}, \boldsymbol{w}) + \boldsymbol{\epsilon} \right) |\right]<+\infty$ for every $\boldsymbol{z} \in \mathcal{Z}$ and $ \boldsymbol{w} \in\mathcal{W}$.

Given a historical dataset $\mathcal{D}_N = \{(\boldsymbol{z}^i, \boldsymbol{w}^i, \boldsymbol{y}^i)\}_{i=1}^N$ with joint observations of $(\boldsymbol{Z},\boldsymbol{W},\boldsymbol{Y})$, if the ground truth regression function $Q^\ast$ is known, then we can construct the true residuals as $\boldsymbol{\epsilon}^i := \boldsymbol{y}^i - Q^\ast(\boldsymbol{z}^i, \boldsymbol{w}^i), \, \forall i \in [N]$ and build the following full-information decision-dependent SAA
\begin{align}\label{eq:fi-er-dd-saa}
    \psi^\ast_{N,\mathrm{ER}}(\boldsymbol{w})
    := \min_{\boldsymbol{z} \in \mathcal{Z}}
    \left\{
      \upsilon^\ast_N(\boldsymbol{z}, \boldsymbol{w})
    := \frac{1}{N} \sum_{i=1}^N c\Big(\boldsymbol{z},  Q^\ast(\boldsymbol{z}, \boldsymbol{w}) + \boldsymbol{\epsilon}^i \Big) 
    \right\}.
\end{align}
However, in practice, $Q^\ast$ is often unknown. We therefore estimate it using a nonparametric regression model $\hat{Q}_N(\cdot, \cdot)$ trained on the dataset $\mathcal{D}_N$. 
We then compute the empirical residuals $\hat{\boldsymbol{\epsilon}}^i := \boldsymbol{y}^i - \hat{Q}_N(\boldsymbol{z}^i, \boldsymbol{w}^i), \, \forall i \in [N]$. For a new covariate $\boldsymbol{w}$ and any decision $\boldsymbol{z}$, we build uncertainty scenarios by adding these empirical residuals to the point prediction $\hat{Q}_N(\boldsymbol{z}, \boldsymbol{w})$ and construct the following ER-DD-SAA
\begin{align}\label{eq:er-DD-SAA}
    \hat{\psi}_{N,\mathrm{ER}}(\boldsymbol{w})
    := \min_{\boldsymbol{z} \in \mathcal{Z}}
    \left\{
    \hat{  \upsilon}_N(\boldsymbol{z}, \boldsymbol{w})
    := \frac{1}{N} \sum_{i=1}^N c\Big(\boldsymbol{z}, \operatorname{Proj}_{\mathcal{Y}}\big( \hat{Q}_N(\boldsymbol{z}, \boldsymbol{w}) + \hat{\boldsymbol{\epsilon}}^i \big) \Big)
    \right\},
\end{align}
where \(\operatorname{Proj}_{\mathcal{Y}}(\cdot)\) projects each scenario onto the support \(\mathcal{Y}\) to ensure validity. We focus on projections that are MILP-representable, e.g., when $\mathcal{Y}$ is a half space or contains simple box constraints, so that embedding the projection preserves the MILP structure of the resulting reformulation. We denote the optimal solution set of the true problem~\eqref{eq:er-dd-saa-true} as \(\mathcal{F}^\ast(\boldsymbol{w})\), and the optimal solution set to the ER-DD-SAA problem~\eqref{eq:er-DD-SAA} as \(\hat{\mathcal{F}}_{N,\text{ER}}(\boldsymbol{w})\).

\subsection{Nonparametric regression}
We next review the three nonparametric regression models used in this paper.

$k$NN regression~\citep{altman1992introduction} predicts the response at a new input by averaging the responses of its $k$ nearest neighbors in the training dataset, i.e, the $k$NN prediction at a new input $(\boldsymbol{z}, \boldsymbol{w})$ is given by $\hat{Q}_N^{k\mathrm{NN}}(\boldsymbol{z}, \boldsymbol{w}) = \frac{1}{k} \sum_{i \in \mathcal{N}_k(\boldsymbol{z}, \boldsymbol{w})} \boldsymbol{y}^i$,
where $\mathcal{N}_k(\boldsymbol{z}, \boldsymbol{w})$ is the index set of the $k$ nearest neighbors of $(\boldsymbol{z}, \boldsymbol{w})$, defined by
\(
\mathcal{N}_k(\boldsymbol{z}, \boldsymbol{w}) := \left\{ i \in [N] : \sum_{j=1}^{N} \mathbb{I}\big\{ \|(\boldsymbol{z}, \boldsymbol{w}) - (\boldsymbol{z}^j, \boldsymbol{w}^j)\| \leq \|(\boldsymbol{z}, \boldsymbol{w}) - (\boldsymbol{z}^i, \boldsymbol{w}^i)\| \big\} \leq k \right\},
\)
and $\mathbb{I}\{\cdot\}$ is the indicator function.

CART~\citep{breiman2017classification} recursively partitions the feature space into axis-aligned rectangles $\mathcal{B}_r,\ r=1,\ldots, N_R$, called \textit{leaf regions} or \textit{leaf nodes}. The prediction for a new input is the average response of the training samples in the same leaf region: $\hat{Q}_N^{\mathrm{CART}}(\boldsymbol{z}, \boldsymbol{w}) = \sum_{r=1}^{N_R} \hat{y}_r \, \mathbb{I}\big\{ (\boldsymbol{z}, \boldsymbol{w}) \in \mathcal{B}_r \big\}$,
where 
$\hat{y}_r = \frac{1}{|\mathcal{B}_r|} \sum_{i : (\boldsymbol{z}^i, \boldsymbol{w}^i) \in \mathcal{B}_r} \boldsymbol{y}^i$ is the average response over the training samples in region $\mathcal{B}_r$.

ReLU NNs~\citep{goodfellow2016deep} approximate complex nonlinear functions by composing affine transformations with element-wise ReLU activations. Given an input $(\boldsymbol{z}, \boldsymbol{w})$, set $\boldsymbol{h}^{(0)} := [\boldsymbol{z}; \boldsymbol{w}]$. Each hidden layer $l \in [L]$ is computed as
$\boldsymbol{h}^{(l)} = \max\big\{ 0, \, \hat{\boldsymbol{v}}^{(l)} \boldsymbol{h}^{(l-1)} + \hat{\boldsymbol{\alpha}}^{(l)} \big\}$,
and the NN prediction is given by $\hat{Q}_N^{\mathrm{NN}}(\boldsymbol{z}, \boldsymbol{w}) = \hat{\boldsymbol{v}}^{(L+1)} \boldsymbol{h}^{(L)} + \hat{\boldsymbol{\alpha}}^{(L+1)}$, where $\hat{\boldsymbol{v}}^{(l)}$ and $\hat{\boldsymbol{\alpha}}^{(l)}$ are the trained weight matrix and bias vector at layer $l$ for $l\in [L]$.

\paragraph{Notation.}
Throughout the paper, we use bold symbols to denote vectors. For any positive integer \(N\), let \([N]:=\{1,2,\ldots,N\}\). We denote the extended real line by $\overline{\mathbb{R}} \coloneqq \mathbb{R} \cup \{-\infty, +\infty\}$. For any two nonempty sets $A, B \subseteq \mathbb{R}^n$, we define the deviation of $A$ from $B$ as  \( \mathbb{D} (A, B) := \sup_{a \in A}  \text{dist}(a, B)\) where \(\text{dist}(a, B)=\inf_{b \in B} \| a-b\|\).

\section{MIP Formulations for ER-DD-SAA with Nonparametric Learning}
\label{sec:MIP-formulations}
In this section, we propose MIP representations for nonparametric regression models within ER-DD-SAA framework. We present the MIP formulations for ER-DD-SAA with \(k\)NN in Section~\ref{sec:mip_knn}, with CART in Section~\ref{subsec:mip_cart}, and with ReLU NNs in Section~\ref{subsec:mip_nn}, respectively.

\subsection{MIP formulations for ER-DD-SAA with \(k\)NN}\label{sec:mip_knn}
We present two MIP formulations for ER-DD-SAA with \(k\)NN: a pairwise distance comparison formulation in Section~\ref{subsubsec:pairwise_distance} and a bilevel formulation with its equivalent single-level MIP reformulation in Section~\ref{subsubsec:bilevel}, respectively. A third MIP formulation based on distance ranking is presented in Appendix \ref{appendix:ranking}.

\subsubsection{Pairwise distance comparison formulation}\label{subsubsec:pairwise_distance}
Given the dataset \(\mathcal{D}_N=\{(\boldsymbol{z}^i, \boldsymbol{w}^i, \boldsymbol{y}^i)\}_{i=1}^N\),  we define the $L_1$ distance between the query point $(\boldsymbol z,\boldsymbol w)$ and data point $i$ as 
\(s_i := \left\|(\boldsymbol z,\boldsymbol w)-(\boldsymbol z^i,\boldsymbol w^i)\right\|_1,
\, \forall i\in [N].
\)
For simplicity, we use \(L_1\)-norm in this paper; however, the proposed framework can also be applied to other distance metrics that admit an MILP representation (e.g., $L_{\infty}$-norm).
For a fixed $k$, we first develop an MILP formulation that models the \(k\)-nearest-neighbor selection within ER-DD-SAA through pairwise distance comparisons. Specifically, let \( d_{ij} \) be a binary variable that indicates whether data point $i$ is farther than (or the same as) $j$ to the new query point \((\boldsymbol{z}, \boldsymbol{w})\), such that \(d_{ij}=1\) if \(s_i  \ge s_j \), and \(d_{ij}=0\) otherwise. Let \(t_i\) be a binary variable indicating whether historical observation \((\boldsymbol z^i,\boldsymbol w^i)\) belongs to the set of \(k\) nearest neighbors of the query point \((\boldsymbol z,\boldsymbol w)\). Then the
ER-DD-SAA with the \(k\)NN regression model can be recast as
follows: 
\allowdisplaybreaks
\begin{subequations}\label{eq:distance_compare_tie}
\begin{align}
\min_{\boldsymbol d, \boldsymbol{s}, \boldsymbol{t}, \boldsymbol z}\quad & \frac{1}{N}\sum_{i=1}^N
    c\!\left(
        \boldsymbol z,
        \operatorname{Proj}_{\mathcal Y}\!\left(
            \frac{1}{k}\sum_{j=1}^N t_j y^j + \hat{\epsilon}^i
        \right)
    \right) \label{eq:distance_comparison_obj}\\
    \text{s.t.} \quad 
    &\boldsymbol z \in \mathcal Z, \label{eq:decision_feasibility} \\
    & s_i = \|(\boldsymbol{z}^i, \boldsymbol{w}^i) - (\boldsymbol{z}, \boldsymbol{w}) \|_1, \quad \forall i\in [N], \label{eq:random_select_knn_p_norm} \\
    & \sum_{i=1}^N t_i = k, \label{eq:random_select_knn_sum_t} \\
    & M_1^{ij}(d_{ij} - 1) \le s_i  - s_j  \leq M_1^{ij} d_{ij}, \quad \forall i,j \in [N],\ i<j, \label{eq:random_select_knn_strict1} \\
    & k- M_2 t_i \le \sum_{i < j} d_{ij} + \sum_{i > j} (1-d_{ji})+ 1 
\le k+ M_2 (1-t_i) , \quad \forall i \in [N], \label{eq:random_select_knn_strict3} \\  
    & t_i \in \{0,1\}, \quad \forall i \in [N], \label{eq:knn_binary} \\
    & d_{ij} \in \{0,1\}, \quad \forall i, j \in [N], \ i < j. \label{eq:comparison_binary}
\end{align}
\label{eq:distance_compare}
\end{subequations} 
In the objective function \eqref{eq:distance_comparison_obj}, $\frac{1}{k}\sum_j t_j y^j$ denotes the $k$NN point prediction at the query point $(\boldsymbol{z},\boldsymbol{w})$, and $\operatorname{Proj}_{\mathcal{Y}} (\frac{1}{k} \sum_j t_j y^j+ \hat{\epsilon}^i)$ represents the uncertainty scenario $i$ in the ER-DD-SAA framework after adding the empirical residual $\hat{\epsilon}^i$ to the point prediction and projecting it onto the support set $\mathcal{Y}$. Constraints~\eqref{eq:random_select_knn_p_norm} compute the \(L_1\)-norm distance between each historical observation \((\boldsymbol{z}^i, \boldsymbol{w}^i)\) and the query point \((\boldsymbol{z}, \boldsymbol{w})\).  Constraint~\eqref{eq:random_select_knn_sum_t} enforces that exactly \(k\) observations are selected as $k$ nearest neighbors.  Constraints~\eqref{eq:random_select_knn_strict1} model pairwise distance comparisons, where \(d_{ij} = 1\) if \(s_i\ge s_j \), and $d_{ij}=0$ if $s_i\le s_j$. In the case when we have a tie, i.e., \(s_i = s_j\), constraints~\eqref{eq:random_select_knn_strict1} permit either \(d_{ij} = 0\) or \(d_{ij} = 1\). We set $d_{ii}=1$ for each $i\in [N]$. Here, $M_1^{ij}$ represents a sufficiently large constant. By triangle inequality, the tightest valid value can be set to \(M_1^{ij} = \big\|(\boldsymbol{z}^i,\boldsymbol{w}^i) - (\boldsymbol{z}^j,\boldsymbol{w}^j)\big\|_1\). In principle, representing the complete pairwise comparison requires \(N^2\) constraints and \(N^2\) decision variables of $d_{ij}$. However, we know that \(d_{ij}+d_{ji}=1\) for all \(1\le i,j\le N\). Therefore, it suffices to explicitly impose the comparison constraints~\eqref{eq:random_select_knn_strict1} for only one direction of each pair, e.g., for \(1\le i<j\le N\), and define the reverse direction by \(d_{ji}=1-d_{ij}\). In constraints~\eqref{eq:random_select_knn_strict3}, for each \(i \in [N]\), the expression
\(
\sum_{j=1}^N d_{ij}=\sum_{i<j} d_{ij}+\sum_{i>j}(1-d_{ji})+1
\)
counts the number of observations whose distances are no greater than distance \(s_i \), and thus represents the rank of \(s_i \) in the ordered list of distances. Hence, constraints~\eqref{eq:random_select_knn_strict3} enforce that \(t_i=1\) if and only if the rank of observation \(i\) is no greater than \(k\), that is, \(\sum_{j=1}^N d_{ij} \le k\). Here, $M_2$ is another sufficiently large constant and its tightest value can be set to $M_2=N-k$. Constraints \eqref{eq:knn_binary} and \eqref{eq:comparison_binary} enforce that $d_{ij}$ and $t_i$ are binary decision variables.


The final formulation type of \eqref{eq:distance_compare_tie} depends on the specific form of the cost function \(c\). In the single-stage setting, when \(c(\boldsymbol{z}, \boldsymbol{Y})\) is affine in \(\boldsymbol{z}\) for any fixed $\boldsymbol{Y}$ and affine in $ \boldsymbol{Y}$ for any fixed $ \boldsymbol{z}$, the objective~\eqref{eq:distance_comparison_obj} remains MILP-representable and the problem \eqref{eq:distance_compare_tie} is an MILP. For example, in the newsvendor problem with pricing considered in Section \ref{sec:newsvendor}, we have a revenue term $pY$, where $p$ is the continuous price decision and $Y$ is the uncertain demand. This results in a bilinear term $pt_j$, which can be linearized exactly using McCormick envelopes \citep{mccormick1976computability} due to the binary variable $t_j$. We refer to Section \ref{subsec:newsvendor_full_formulation} for the full model of this newsvendor problem. In two-stage stochastic MILPs with RHS uncertainty, $\operatorname{Proj}_{\mathcal{Y}} (\frac{1}{k} \sum_j t_j y^j+ \hat{\epsilon}^i)$ will show up on the right-hand side of the second-stage constraints and the resulting formulation can also be recast as an MILP. We refer to Section \ref{subsec:facility_formulation} for a full reformulation of the two-stage facility location example.

\paragraph{Discussion: objective-driven tie breaking.} When multiple historical observations are equidistant from the query point, the \(k\)NN selection may not be unique. For example, when $s_i=s_j$, both data points $i$ and $j$ can be selected as one of the $k$ nearest neighbors, i.e., both $t_i$ and $t_j$ could be set to $1$. In this case, the optimization model implicitly breaks ties in favor of observations that lead to a better objective value. 

\subsubsection{Bilevel formulation}\label{subsubsec:bilevel}
Model~\eqref{eq:distance_compare_tie} involves a quadratic number of variables and constraints, which limits its tractability for large-scale problems. To improve scalability, we propose the following equivalent but more compact bilevel formulation 
\begin{subequations}\label{eq:bilevel}
   \begin{align}
    \min_{\boldsymbol z ,\ \boldsymbol s}\quad
    & \frac{1}{N}\sum_{i=1}^N
    c\!\left(
        \boldsymbol z,
        \operatorname{Proj}_{\mathcal Y}\!\left(
            \frac{1}{k}\sum_{j=1}^N t_j y^j + \hat{\epsilon}^i
        \right)
    \right) \\
    \text{s.t.}\quad
    & \eqref{eq:decision_feasibility}--\eqref{eq:random_select_knn_p_norm}, \\
    & \boldsymbol t \in \arg\min_{\boldsymbol t}
    \left\{
        \sum_{j=1}^N s_j  t_j :
        \sum_{j=1}^N  t_j = k,\ 
         t_j \in \{0,1\},\ \forall j \in [N]
    \right\}. \label{eq:lower_level}
\end{align} 
\end{subequations}
Here, the upper-level problem imposes the ER-DD-SAA objective with variable-defining constraints \eqref{eq:decision_feasibility}--\eqref{eq:random_select_knn_p_norm}, and the lower-level problem~\eqref{eq:lower_level} identifies the \(k\) nearest neighbors by selecting \(k\) historical observations with the smallest total distance to the query point \(\sum_{j=1}^N s_j t_j\). For a given \(\boldsymbol{s}\), the lower-level problem~\((\ref{eq:lower_level})\) is a 0-1 knapsack problem whose associated constraint matrix is totally unimodular \citep{schrijver1998theory}. Therefore, when the binary variables \(t_j \in \{0,1\}\) are relaxed to \(t_j\in[0,1]\), the resulting LP relaxation admits integral optimal solutions at extreme points. Hence, the binary lower-level problem~\eqref{eq:lower_level} is equivalent to its LP relaxation \(\boldsymbol t \in \arg\min_{\boldsymbol t}
    \left\{
        \sum_{j=1}^N s_j  t_j :
        \sum_{j=1}^N  t_j = k,\ 
         t_j \in [0,1],\ \forall j \in [N]
    \right\}. \)


Let \(\boldsymbol{\pi}\ge 0\) and \(\beta\) denote the dual variables associated with the constraints \(t_j \le 1, \, \forall j\in[N]\), and \(\sum_{j=1}^N t_j = k\), respectively. By strong duality, the LP relaxation can be recast as its dual: $\max_{\boldsymbol{\pi} \ge 0,\ \beta} \ \left\{ -\sum_{j=1}^N \pi_j + k\beta \ : \ \beta - \pi_j \le s_j, \ \forall j \in [N] \right\}$.
As a result, the bilevel problem~\eqref{eq:bilevel} can be reformulated as the following single-level problem:
\allowdisplaybreaks
\begin{subequations}\label{eq:single_level}
   \begin{align}
    \min_{\boldsymbol z, \boldsymbol{s}, \boldsymbol t, \beta, \boldsymbol{\pi}}\quad
    & \frac{1}{N}\sum_{i=1}^N
    c\!\left(
        \boldsymbol z,
        \operatorname{Proj}_{\mathcal Y}\!\left(
            \frac{1}{k}\sum_{j=1}^N t_j y^j + \hat{\epsilon}^i
        \right)
    \right) \\
    \text{s.t.}\quad
    & \eqref{eq:decision_feasibility}--\eqref{eq:random_select_knn_p_norm}, \\
        & 0 \le t_j \le 1, \quad \forall j \in [N], \label{eq:bilevel_t_range} \\
    & \sum_{j=1}^N t_j = k, \label{eq:primal_feasible_sum}\\
    & \beta - \pi_j \le s_j, \quad \forall j \in [N], \label{eq:bilevel_dual_constraints}\\
        & \boldsymbol{\pi} \ge 0, \label{eq:bilevel_pi_constraints} \\
    & \sum_{j=1}^N s_j t_j \le -\sum_{j=1}^N \pi_j +  k \beta, \label{eq:strong_duality}
    \end{align} 
\end{subequations}
where \eqref{eq:bilevel_t_range} and \eqref{eq:primal_feasible_sum} enforce the primal feasibility, \eqref{eq:bilevel_dual_constraints} and \eqref{eq:bilevel_pi_constraints} enforce the dual feasibility, and \eqref{eq:strong_duality} enforces the strong duality. The single-level reformulation above still contains bilinear terms \(s_j t_j, \, \forall j \in [N]\). Since \eqref{eq:lower_level} is equivalent to \eqref{eq:bilevel_t_range}--\eqref{eq:strong_duality}, the solution $\boldsymbol t^*$ of \eqref{eq:single_level} will achieve binary values at optimality. We may therefore replace \eqref{eq:bilevel_t_range} with \(t_j\in\{0,1\}\), \(j\in[N]\) in the single-level reformulation \eqref{eq:single_level} without changing its optimal value. The resulting bilinear terms can then be linearized exactly using McCormick envelopes \citep{mccormick1976computability}. We denote $\overline{s}_i$ as a valid upper bound on $s_i, \, \forall i \in [N]$. This yields the following equivalent MILP reformulation:
\begin{subequations}\label{eq:single_level_binary}
   \begin{align}
    \min_{\boldsymbol z, \boldsymbol{s}, \boldsymbol t, \beta, \boldsymbol{\pi},\boldsymbol{\gamma}}\quad
    & \frac{1}{N}\sum_{i=1}^N
    c\!\left(
        \boldsymbol z,
        \operatorname{Proj}_{\mathcal Y}\!\left(
            \frac{1}{k}\sum_{j=1}^N t_j y^j + \hat{\epsilon}^i
        \right)
    \right) \\
    \text{s.t.}\quad
    & \eqref{eq:decision_feasibility}--\eqref{eq:random_select_knn_p_norm},\\
    & \eqref{eq:primal_feasible_sum} -- \eqref{eq:bilevel_pi_constraints}, \label{eq:bilevel-1}\\
    & 0\le \gamma_{j} \le s_j, \quad \forall j \in [N],\label{eq:single_level_binary_1}\\
    &s_j + \overline{s}_j (t_j-1)\le \gamma_{j} \le \overline{s}_j t_j,\quad \forall j \in [N],\\
    & \sum_{j=1}^N \gamma_j \le -\sum_{j=1}^N \pi_j +  k \beta, \label{eq:single_level_binary_2} \\
    & t_j \in \{0,1\}, \quad \forall j \in [N].\label{eq:bilevel-last}
    \end{align} 
\end{subequations}
Compared to Model \eqref{eq:distance_compare_tie}, the number of decision variables and number of constraints in Model~\eqref{eq:single_level_binary} grow only linearly with the sample size $N$, which is more computationally tractable.

\subsection{ER-DD-SAA with CART}\label{subsec:mip_cart}
Once trained, the CART model assigns any query point \((\boldsymbol z,\boldsymbol w)\) to a unique leaf region \(r\in [N_R]\), and the associated prediction is given by the average response of the training samples contained in that leaf. We denote the leaf region \(r\) as $\mathcal{B}_r
=
\left\{
(\boldsymbol z,\boldsymbol w)\in\mathbb R^{d_z}\times\mathbb R^{d_w}
:
\boldsymbol a_z^r < \boldsymbol z \le \boldsymbol b_z^r,\;
\boldsymbol a_w^r < \boldsymbol w \le \boldsymbol b_w^r
\right\}$,
where \(\boldsymbol a_z^r,\boldsymbol b_z^r\in\overline{\mathbb{R}}^{d_z}\) and
\(\boldsymbol a_w^r,\boldsymbol b_w^r\in\overline{\mathbb{R}}^{d_w}\)
are the componentwise lower and upper bounds of $\boldsymbol{z}$ and $\boldsymbol{w}$ in leaf node \(r\) obtained from training. Given a realization of the covariate \(\boldsymbol w\), only those leaf nodes whose covariate intervals contain \(\boldsymbol w\) remain feasible. Accordingly, we may consider only a subset of candidate leaf nodes, denoted by
$\mathcal{N}^{\boldsymbol w}
:=
\left\{
r \in [N_R]
: \exists \boldsymbol{z}, 
(\boldsymbol z, \boldsymbol w)\in \mathcal{B}_r\}=\{r \in [N_R]
: \boldsymbol{w} \in (\boldsymbol a_w^r,\boldsymbol b_w^r]
\right\}$.
For each \(r\in\mathcal{N}^{\boldsymbol w}\), the associated region in the decision space is given by \((\boldsymbol a_z^r,\boldsymbol b_z^r]\). Therefore, once \(\boldsymbol w\) is fixed, determining whether $(\boldsymbol z,\boldsymbol{w})\in \mathcal{B}_r$ is equivalent to determining whether \(\boldsymbol z \in (\boldsymbol a_z^r,\boldsymbol b_z^r]\) for each $r\in\mathcal{N}^{\boldsymbol{w}}$. Let $N^{\boldsymbol{w}}:=|\mathcal{N}^{\boldsymbol{w}}|$ denote the number of leaf regions in the reduced set $\mathcal{N}^{\boldsymbol{w}}$. 

To embed the CART predictor into the optimization model, let \(R_r\in\{0,1\}\) denote whether leaf node \(r\in\mathcal{N}^{\boldsymbol w}\) contains the query point $(\boldsymbol{z},\boldsymbol{w})$, that is, \(R_r=1\) if and only if \(\boldsymbol z\) satisfies the bounds of region \(r\), namely, \(\boldsymbol a_z^r < \boldsymbol z \le \boldsymbol b_z^r\). For each leaf node \(r\), we further introduce binary decision variables \(l_r, u_r \in \{0,1\}\) to indicate whether \(\boldsymbol z\) satisfies the corresponding lower and upper bounds of region \(r\), respectively. Let \(\hat{y}^r\) denote the prediction associated with leaf node \(r\). The resulting exact MIP formulation for the ER-DD-SAA with CART is given by
\allowdisplaybreaks
\begin{subequations}
\begin{align}\min_{\boldsymbol{l},\boldsymbol{u},\boldsymbol{R}, \boldsymbol{z}} \quad & \frac{1}{N}\sum_{i=1}^N
    c\!\left(
        \boldsymbol z,
        \operatorname{Proj}_{\mathcal Y}\!\left(
            \sum_{r\in \mathcal{N}^{\boldsymbol{w}}} R_r \hat{y}^r + \hat{\epsilon}^i
        \right)
    \right) \\
    \text{s.t.} \quad 
    & \boldsymbol M (l_{r} - 1 ) + \boldsymbol \delta \le \boldsymbol{z} - \boldsymbol{a}_z^r \leq  \boldsymbol M l_{r}, \quad  \forall r \in \mathcal{N}^{\boldsymbol w},\ \label{eq:ind_region_lowerbound1} \\
    & -\boldsymbol M (1- u_{r}) \le \boldsymbol{b}_z^r - \boldsymbol{z} \leq \boldsymbol M u_{r} - \boldsymbol \delta, \quad  \forall r \in \mathcal{N}^{\boldsymbol w}, \label{eq:ind_region_upperbound1} \\
    &l_{r} + u_{r} \geq 2R_r, \quad  \forall r \in \mathcal{N}^{\boldsymbol w}, \label{eq:box-activation} 
    \\
    & \sum_{r\in \mathcal{N}^{\boldsymbol w}} R_r = 1, \label{eq:unique-box}\\
    & \boldsymbol z \in \mathcal Z, \label{eq:cart_decision_feasiblity}\\
    & l_{r}, u_{r}, R_r \in \{0, 1\}, \quad \forall r\in \mathcal{N}^{\boldsymbol w}, \label{eq:cart_binary_decision}
\end{align}
\label{eq:er_dd_saa_cart}
\end{subequations}
where constraints~\eqref{eq:ind_region_lowerbound1} enforce that $l_r = 1$ if and only if the decision vector satisfies the constraint $\boldsymbol{z} > \boldsymbol{a}_z^r$ componentwise. Similarly, Constraints~\eqref{eq:ind_region_upperbound1} ensure that $u_r = 1$ if and only if the decision vector satisfies $\boldsymbol{z} \le \boldsymbol{b}_z^r$ componentwise. Here, $\boldsymbol{M}$ is a sufficiently large vector, and $\boldsymbol{\delta}>0$ is a sufficiently small positive vector to enforce the strict inequalities. Constraints~\eqref{eq:box-activation} couple these bound-satisfaction indicators with the region-selection variable $R_r$, dictating that $R_r = 1$ only when both the lower and upper bounds of region $r$ are satisfied simultaneously, i.e., $l_r=u_r=1$. Furthermore, Constraint~\eqref{eq:unique-box} restricts the model to select exactly one region for any feasible decision $\boldsymbol{z}$. Consequently, Constraints~\eqref{eq:box-activation} and~\eqref{eq:unique-box} jointly guarantee the activation of the unique region containing $\boldsymbol{z}$. The final formulation type of \eqref{eq:er_dd_saa_cart} also depends on the specific cost function $c$. Since the CART prediction term contains binary variables $R_r$, we obtain the same conclusions as in the $k$NN case, where the resulting formulation can be recast as an MILP with both objective uncertainty and RHS uncertainty. We refer to Appendix~\ref{sec:full_formulation} for full reformulations under CART. For the special case when the decision is single-dimensional (\(d_z=1\)), we also provide an equivalent formulation in Appendix~\ref{sec:er_dd_saa_cart_omitted}.

\subsection{ER-DD-SAA with ReLU NNs}\label{subsec:mip_nn}
In this section, we consider a feedforward neural network trained prior to solving the downstream optimization problem. Let $\{(\hat{\boldsymbol{v}}^{(l)},\hat{\boldsymbol{\alpha}}^{(l)})\}_{l=1}^{L+1}$ denote the estimated weight matrices and bias vectors, trained on the historical dataset $\mathcal{D}_N=\{(\boldsymbol{z}^i, \boldsymbol{w}^i, \boldsymbol{y}^i)\}_{i=1}^N$. These parameters can be learned by minimizing mean squared error via gradient-based stochastic optimization algorithms. 

The resulting ER-DD-SAA problem is obtained by embedding the trained ReLU NNs into the optimization model as follows:
\allowdisplaybreaks
\begin{subequations}
\begin{align}
\min_{\hat{y},\boldsymbol{h},\boldsymbol{z}}\quad
& \frac{1}{N}\sum_{i=1}^N
    c\!\left(
        \boldsymbol{z},
        \operatorname{Proj}_{\mathcal{Y}}\!\left(
            \hat{y} + \hat{\boldsymbol{\epsilon}}_i
        \right)
    \right) \\
\text{s.t.}\quad &\boldsymbol{z} \in \mathcal{Z}, \\
& \boldsymbol{h}^{(0)} = [\boldsymbol{z};\boldsymbol{w}], \label{eq:nn_first_layer}\\
& \boldsymbol{h}^{(l)} = \max \big\{0, \hat{\boldsymbol{v}}^{(l)}\boldsymbol{h}^{(l-1)} + \hat{\boldsymbol{\alpha}}^{(l)} \big\},
\quad \forall l \in [L], \label{eq:relu_activation} \\
& \hat{y} = \hat{\boldsymbol{v}}^{(L+1)}\boldsymbol{h}^{(L)} + \hat{\boldsymbol{\alpha}}^{(L+1)}. \label{eq:nn_output_layer}
\end{align}
\label{eq:nn_erddsaa}
\end{subequations}
Let \(N_l\) denote the number of neurons for each layer. For each neuron $n \in [N_l]$ in layer $l \in [L]$, the ReLU activation function in \eqref{eq:relu_activation} can be modeled as \citep{badilla2023computational}: 
\begin{subequations}\label{eq:relu_milp}
\begin{align}
& h_n^{(l)} - \eta_n^{(l)} = \hat{\boldsymbol{v}}_n^{(l)}\boldsymbol{h}^{(l-1)} + \hat{\alpha}_n^{(l)}, \label{eq:relu_linear_1} \\
& 0 \le h_n^{(l)} \le U_n^{(l)}   \kappa_n^{(l)}, \label{eq:relu_linear_2} \\
& 0 \le \eta_n^{(l)} \le -L_n^{(l)} \big(1-  \kappa_n^{(l)}\big), \label{eq:relu_linear_3} \\
&   \kappa_n^{(l)} \in \{0,1\}, \label{eq:relu_linear_4}
\end{align}
\end{subequations}
where $h_n^{(l)}$ and $\eta_n^{(l)}$ denote the positive and negative components of the \(\hat{\boldsymbol{v}}_n^{(l)}\boldsymbol{h}^{(l-1)} + \hat{\alpha}_n^{(l)}\), $U_n^{(l)}$ and $L_n^{(l)}$ are valid upper and lower bounds on \(\hat{\boldsymbol{v}}_n^{(l)}\boldsymbol{h}^{(l-1)} + \hat{\alpha}_n^{(l)}\), and $  \kappa_n^{(l)}$ indicates whether \(\hat{\boldsymbol{v}}_n^{(l)}\boldsymbol{h}^{(l-1)} + \hat{\alpha}_n^{(l)}\) is non-negative. The complexity of \eqref{eq:nn_erddsaa} again depends on the specific cost function \(c\). Unlike the $k$NN and CART models, when uncertainty enters the objective function through products with other decision variables, the resulting formulation may contain bilinear terms involving the continuous prediction variable $\hat{y}$. If $\hat{y}$
 is multiplied only by binary decision variables, these bilinear terms can be linearized exactly using McCormick envelopes, allowing the overall problem to be reformulated as an MILP. In contrast, if $\hat{y}$
 appears in bilinear products with other continuous decision variables, exact linearization is generally not possible, and the resulting formulation becomes an MINLP. For example, in the newsvendor problem with pricing presented in Section~\ref{sec:newsvendor}, the cost function $c(p,q,Y)= f q - pY + h(q-Y)^+ + b(Y-q)^+$ contains the bilinear term $p\hat{y}$. Since both $p$ and $\hat{y}$ are continuous decision variables, the resulting model \eqref{eq:nn_erddsaa} is an MINLP. The complete MINLP formulation is presented in Model \eqref{eq:newsvendor_nn} in Appendix \ref{subsec:newsvendor_full_formulation}. When uncertainty appears on the RHS of the constraints (e.g., in two-stage facility location problems presented in Section~\ref{sec:facility_location_problem}), the resulting model \eqref{eq:nn_erddsaa} can be recast as an MILP. Please refer to Model \eqref{eq:er_dd_saa_facility_full_nn} for the full MILP representation.

\subsection{Comparisons between different nonparametric regression models}
We compare the formulation size of the resulting ER-DD-SAA problem under different nonparametric regression models in Table~\ref{tab:complexity_comparison}.
As shown in Table~\ref{tab:complexity_comparison}, the sizes of all \(k\)NN-based formulations depend on the sample size \(N\). Even the most compact bilevel reformulation (i.e.  Model~\eqref{eq:single_level_binary}) scales linearly with \(N\). In contrast, the sizes of the CART and ReLU NNs formulations depend solely on the number of leaf nodes \(N^{\boldsymbol{w}}\) and the total number of hidden neurons \(\sum_{l=1}^L N_l\), respectively. In practice, the sizes of the leaf nodes and neurons are typically orders of magnitude smaller than the sample size \(N\). Consequently, \(k\)NN-based formulations are the most computationally prohibitive ones (especially formulation \eqref{eq:distance_compare_tie}) among all three nonparametric regression models. Motivated by this computational bottleneck, Section~\ref{sec:knn_tractability} focuses on developing tailored decomposition algorithm to improve the tractability of \(k\)NN-based formulations.

\begin{table}[htbp]
\centering
\caption{Comparison of formulation complexity across different regression models.}
\label{tab:complexity_comparison}
\begin{tabular}{lcccc}
\toprule
& \(k\)NN~\eqref{eq:distance_compare_tie}
& \(k\)NN~\eqref{eq:single_level_binary}
& CART~\eqref{eq:er_dd_saa_cart}
& ReLU NNs~\eqref{eq:nn_erddsaa}
\\ \midrule
\# binary variables
& \(\mathcal{O}(N^2)\)
& \(\mathcal{O}(N)\)
& \(\mathcal{O}(N^{\boldsymbol{w}})\)
& \(\mathcal{O}\bigl(\sum_{l=1}^L N_l\bigr)\)
\\
\# constraints
& \(\mathcal{O}(N^2)\)
& \(\mathcal{O}(N)\)
& \(\mathcal{O}(N^{\boldsymbol{w}})\)
& \(\mathcal{O}\bigl(\sum_{l=1}^L N_l\bigr)\)
\\ \bottomrule
\end{tabular}
\end{table}


\section{Decomposition Algorithm for Two-Stage ER-DD-SAA with \(k\)NN} \label{sec:knn_tractability} 

In this section, we develop a tailored decomposition algorithm for a general two-stage stochastic programming problem with decision-dependent uncertainty of the following form:
\allowdisplaybreaks
\begin{align}\label{eq:two_stage_problem}
    \min_{\boldsymbol{z} \in \mathcal{Z}} \quad \boldsymbol{f}^\top \boldsymbol{z} + \mathbb{E}[h(\boldsymbol{z}, \boldsymbol{Y}(\boldsymbol{z}, \boldsymbol{w}))|Z=\boldsymbol{z}, W=\boldsymbol{w}], 
\end{align}
where the second-stage recourse function $h(\boldsymbol{z},\boldsymbol{Y}(\boldsymbol{z}, \boldsymbol{w}))$ is given by:
\allowdisplaybreaks
\begin{subequations}\label{eq:recourse_function}
\begin{align} 
h(\boldsymbol{z}, \boldsymbol{Y}(\boldsymbol{z}, \boldsymbol{w})) := \min_{\boldsymbol{x}} \quad & \boldsymbol{g}^{\top}\boldsymbol{x} \\ 
\text{s.t.} \quad & \boldsymbol{A}\boldsymbol{x} \ge H(\boldsymbol{z}) + \boldsymbol{T}\boldsymbol{Y}(\boldsymbol{z}, \boldsymbol{w}), \\
& \boldsymbol{x} \in \mathbb{R}_{+}^{d_x}.
\end{align} 
\end{subequations}
Here, $\boldsymbol{z}\in\mathcal{Z}$ denotes the first-stage decision with cost vector $ \boldsymbol f\in\mathbb{R}^{d_z}$, $\boldsymbol{x}\in\mathbb{R}_+^{d_x}$ denotes the second-stage decision with cost vector $ \boldsymbol g\in\mathbb{R}^{d_x}$, $\boldsymbol Y(\boldsymbol z,\boldsymbol w)$ is the uncertainty that depends on first-stage decision $\boldsymbol z$ and covariate $\boldsymbol w$, and $H(\boldsymbol{z})$ is affine in first-stage decision $ \boldsymbol z$. Note that we only require the second-stage decision variables $\boldsymbol{x}$ to be continuous, while the first-stage decision $\boldsymbol{z}$ can be either continuous or discrete. For simplicity of presentation, we assume $\mathcal{Y}=\mathbb{R}^{d_y}$ and therefore do not consider the projection step in this section. To ensure the problem \eqref{eq:two_stage_problem} is valid, we assume relatively complete recourse and sufficiently expensive recourse as follows:
\begin{assumption} \label{ass:complete_recourse}
    For every feasible first-stage decision $\boldsymbol z\in \mathcal{Z}$ and every realization of $\boldsymbol{Y}$, the second-stage problem~\eqref{eq:recourse_function} is feasible.
\end{assumption}
\begin{assumption}\label{ass:expensive_recourse}
    For every feasible first-stage decision $ \boldsymbol{z}\in\mathcal{Z}$ and every realization of $\boldsymbol{Y}$, the dual of the second-stage problem \eqref{eq:recourse_function} is feasible.
\end{assumption}

Assumptions 1 and 2 ensure that $-\infty< h(\boldsymbol z, \boldsymbol Y(\boldsymbol z,\boldsymbol w))< +\infty$. Since $k$NN's formulation \eqref{eq:distance_compare_tie} is the most computationally challenging one, we focus on developing a decomposition algorithm for two-stage ER-DD-SAA with formulation \eqref{eq:distance_compare_tie} in the following form
\begin{subequations}\label{eq:two_stage_problem_saa}
    \begin{align}
           \min_{\boldsymbol{z}, \boldsymbol{d}, \boldsymbol{s},
           \boldsymbol{t}, \boldsymbol{x}} \quad & \boldsymbol{f}^\top \boldsymbol{z} + \frac{1}{N} \sum_{i=1}^N \boldsymbol{g}^T \boldsymbol{x}_i \\
    \text{s.t.} \quad & \eqref{eq:decision_feasibility}--\eqref{eq:comparison_binary}, \\
    & \boldsymbol{A} \boldsymbol{x}_i \ge H(\boldsymbol{z})+ \boldsymbol{T} \left(\frac{1}{k} \sum_{j=1}^N t_j y^j + \hat{\epsilon}^i\right), \quad \forall i \in [N], \label{eq:recourse_function_constraints} \\
    & \boldsymbol x_i\in \mathbb{R}_+^{d_x}, \quad \forall i\in [N].\label{eq:recourse_function_x}
    \end{align}
\end{subequations}
The main computational difficulty comes from two sources: the large number of constraints required to characterize the \(k\)NN selection in constraints \eqref{eq:decision_feasibility}--\eqref{eq:comparison_binary}, and the large number of scenarios due to empirical residuals in constraints \eqref{eq:recourse_function_constraints}. To address these challenges, this section develops a unified decomposition algorithm named Bender's Decomposition with Constraint Generation (BD-CG), where we treat $\boldsymbol{z},\boldsymbol{d}, \boldsymbol{s}, \boldsymbol{t}$ as first-stage decision variables subject to constraints \eqref{eq:decision_feasibility}--\eqref{eq:comparison_binary} and $\boldsymbol{x}$ as second-stage decision variables subject to constraints \eqref{eq:recourse_function_constraints}--\eqref{eq:recourse_function_x}. Our approach dynamically adds violated first-stage constraints to identify the exact $k$NN neighborhood, while using Bender's optimality cuts to approximate the second-stage recourse function.

At iteration \(m\), the relaxed master problem (RMP) is defined as follows: \allowdisplaybreaks 
\begin{subequations}\label{eq:master_problem} 
\begin{align} 
\min_{\boldsymbol z, \boldsymbol{d}, \boldsymbol{s}, \boldsymbol{t}, \theta} \quad & \boldsymbol{f}^{\top}\boldsymbol{z} + \theta \\ 
\text{s.t.} \quad & \eqref{eq:decision_feasibility}--\eqref{eq:random_select_knn_sum_t}, \eqref{eq:knn_binary}, \label{eq:algorithm_constraint_1}\\ 
& d_{ij} + d_{ji}=1, \quad \forall i \in I,\ j \in J, \label{eq:alg_constraint1}\\ & M_1^{ij}(d_{ij}-1) \le s_i - s_j \le M_1^{ij} d_{ij}, \quad \forall i \in I,\ j \in J, \label{eq:alg_constraint2}\\ 
& k-M_2 t_i \le \sum_{j \in (I \cup J)\setminus\{i\}} d_{ij} + 1 \le k+M_2(1-t_i), \quad \forall i \in I \cup J,\\ 
& d_{ij}\in\{0,1\},\ \forall i,j\in I\cup J,\\
& \theta \ge L_{\theta}, \label{eq:lower_bound_theta} \\
& \theta \ge \Theta^{\ell}(\boldsymbol z, \boldsymbol t), \quad \forall \ell \in [m-1]. \label{eq:Bender's_cuts}
\end{align} 
\end{subequations}
Here, instead of enforcing the pairwise distance comparison constraints \eqref{eq:alg_constraint2} for all $i,j\in [N]$, we start with a small subset $i\in I\subset[N],\ j\in J\subset[N]$, which forms a relaxation of the first-stage feasible region. At the same time, we use \(\theta\) to approximate the expected second-stage recourse function, and \(\{\theta \ge \Theta^{\ell}(\boldsymbol z, \boldsymbol t)\}_{\ell=1}^m\) denotes the set of Bender's optimality cuts generated up to iteration \(m\), where we initialize the algorithm with an empty cut set. The constant \(L_{\theta}\) is a valid lower bound on the expected recourse function and is included to ensure that the RMP is bounded at initialization.


Given an optimal solution $(\boldsymbol z^m, \boldsymbol t^m)$ of the RMP at iteration $m$, we solve the following subproblem, which is a dual problem of \eqref{eq:recourse_function} under scenario $i$:
\allowdisplaybreaks
\begin{subequations}\label{eq:sub_problem}
\begin{align}
    G^i (\boldsymbol z^m, \boldsymbol t^m) 
    := 
    \max_{\boldsymbol{\pi}_i} \quad 
    & \boldsymbol{\pi}_i^{\top} \left(
        H (\boldsymbol z^m) 
        + 
        \boldsymbol{T}\left( 
            \frac{1}{k} \sum_{j=1}^N t^m_j y^j + \hat{\epsilon}^i 
        \right) 
    \right) \\
    \text{s.t.} \quad 
    & \boldsymbol{A}^{\top}\boldsymbol{\pi}_i \le \boldsymbol{g},\ \boldsymbol{\pi}_i \ge \boldsymbol{0}.
\end{align}
\end{subequations}
Let $\boldsymbol{\pi}_i^m\ge 0$ denote an optimal solution of \eqref{eq:sub_problem}. The resulting Bender's optimality cut is given by: 
\begin{equation}\label{eq:dual_cut} 
\theta \ge \Theta^m(\boldsymbol z, \boldsymbol t) := \frac{1}{N} \sum_{i=1}^N \left(\boldsymbol{\pi}_i^m\right)^{\top} \left[  H (\boldsymbol z) 
        + 
        \boldsymbol{T} \left( \frac{1}{k} \sum_{j=1}^N t_j y^j + \hat{\epsilon}^i \right) \right]. 
\end{equation}
Note that although we present the cut in the above single-cut formulation, a multi-cut variant can also be used here. We initialize the algorithm with \(I=\emptyset\), \   \(J=\emptyset\). At each iteration \(m\), we solve the RMP~\eqref{eq:master_problem} and obtain an optimal solution \((\boldsymbol z^m,\boldsymbol d^m,\boldsymbol s^m,\boldsymbol t^m,\theta^m)\), together with the RMP-selected $k$NN set \(F^m:=\{i:t_i^m=1\}\). Because the RMP contains only a subset of the original distance comparison constraints, its objective value serves as a lower bound to the original problem, and the selected $k$NN set \(F^m\) may not coincide with the true \(k\) nearest neighbors for \(\boldsymbol z^m\). Given the current solution \(\boldsymbol z^m\), we identify the true \(k\)NN set \(S^m \coloneqq \{i: i \in \mathcal{N}_k(\boldsymbol{z}^m, \boldsymbol{w})\}\). We then update the first-stage constraints by setting \(I \leftarrow I \cup S^m\) and \(J \leftarrow J \cup F^m\). Given the current candidate solution \((\boldsymbol z^m,\boldsymbol t^m)\), we solve the subproblem~\eqref{eq:sub_problem} for each \(i \in [N]\), generate the Bender's optimality cut~\eqref{eq:dual_cut}, and add it to the RMP~\eqref{eq:master_problem}. We also compute a candidate upper bound by evaluating the recourse function using the true \(k\)NN prediction induced by \(S^m\). We repeatedly solve the updated RMP until convergence. The algorithm terminates once the RMP-selected $k$NN set is consistent with the true $k$NN set, i.e., \(F^m = S^m\), and the newly generated Bender's optimality cut is satisfied. Note that when there is a tie, $S^m$ may contain more than $k$ points if multiple points are equidistant to the query point and can all be served as a $k$NN. In this case, we require $F^m \subseteq S^m$.  Alternatively, one can also terminate the algorithm when $F^m\subseteq S^m$ and the optimality gap between the lower bound and upper bound falls below a predefined tolerance. The detailed procedure is presented in Algorithm~\ref{alg:row_generation_knn}.
\begin{algorithm}[!htbp]
\caption{BD-CG algorithm for ER-DD-SAA with \(k\)NN under formulation \eqref{eq:distance_compare_tie}}
\label{alg:row_generation_knn}
\begin{algorithmic}[1]
\State Initialization: \(I \leftarrow \emptyset\), \(J \leftarrow \emptyset\), \(m \leftarrow 1\), \(LB \leftarrow -\infty\), \(UB \leftarrow +\infty\), and \(\texttt{converged}\leftarrow \texttt{False}\).

\While{\(\texttt{converged}=\texttt{False}\) \textbf{and} \(m<T_{\max}\)}

    \State Solve RMP~\eqref{eq:master_problem} to
    obtain an optimal solution \((\boldsymbol z^m,\boldsymbol d^m,\boldsymbol s^m,\boldsymbol t^m,\theta^m)\) and the RMP-selected $k$NN set \(F^m:=\{i:t_i^m=1\}\). Set \(LB=\) RMP's objective value.
    
    \State Given \(\boldsymbol z^m\) and  \(\boldsymbol w\), identify the true $k$NN set \(S^m \coloneqq \{i: i \in \mathcal{N}_k(\boldsymbol{z}^m, \boldsymbol{w})\}\). Define the indicator vector \(\boldsymbol{\tau}^m \in \{0, 1\}^N\) such that \(\tau_i^m = 1\) if \(i \in S^m\), and \(\tau_i^m = 0\) otherwise.

    \State Update the sets:
    \(
        I \leftarrow I\cup S^m,
        \,
        J \leftarrow J\cup F^m.
    \)

    \State Given \((\boldsymbol z^m, \boldsymbol{t}^m)\) , solve subproblem~\eqref{eq:sub_problem} for each $i\in [N]$, and add the following Bender's optimality cut~\eqref{eq:dual_cut} to the master problem~\eqref{eq:master_problem}: $\theta \ge \Theta^m(\boldsymbol z, \boldsymbol t)$.

    \State Compute the candidate upper bound
    \(
        UB^{\mathrm{cand}}
        =
        \boldsymbol f^\top \boldsymbol z^m
        +
        \frac{1}{N} \sum_{i=1}^NG^i(\boldsymbol z^m, \boldsymbol{\tau}^m).
    \)
    \State Update
    \(
        UB \leftarrow \min\{UB,UB^{\mathrm{cand}}\}
    \) and compute the optimality gap:
    \(
        \mathrm{Gap}
        =
        (UB-LB)/|UB|.
    \)

    \If{\(F^m \subseteq S^m\) \textbf{and} $\theta^m \ge \Theta^m(\boldsymbol z^m, \boldsymbol t^m)$}
        \State Set \(\texttt{converged}\leftarrow \texttt{True}\).
    \EndIf
    \State Set \(m \leftarrow m+1\).
\EndWhile

\State Return the incumbent solution \((\boldsymbol z^m,\boldsymbol d^m,\boldsymbol s^m,\boldsymbol t^m,\theta^m)\).
\end{algorithmic}
\end{algorithm}

We next establish the convergence of the proposed BD-CG algorithm. When the algorithm fails to meet the exact termination criteria, it either adds the violated constraints to the master problem by updating sets \(I\) and \(J\) or generates a Bender's optimality cut. These new constraints prune the current infeasible or suboptimal solution. Since the total number of pairwise distance comparison constraints and Bender's cuts is finite, the algorithm is guaranteed to converge to the global optimum in a finite number of iterations, as shown in the following theorem. All the omitted proofs are shown in Appendix \ref{proof:algorithm_convergence}.

\begin{theorem}\label{thm:alg_convergence}
Suppose Assumptions~\ref{ass:complete_recourse} and~\ref{ass:expensive_recourse} hold, 
\(\mathcal Z\) is nonempty and compact, and the
RMP~\eqref{eq:master_problem} is solved to
global optimality at each iteration. The
proposed BD-CG algorithm converges to a globally optimal solution in finitely
many iterations.
\end{theorem}


\section{Statistical Guarantee: Consistency and Asymptotic Optimality}\label{sec:consistency}
In this section, we establish the asymptotic optimality and consistency of the ER-DD-SAA problem~\eqref{eq:er-DD-SAA} under $k$NN, CART, and ReLU NNs. To establish these theoretical guarantees, we first introduce several assumptions.

\begin{assumption}\label{assumption:liptschitz}
For each \(\boldsymbol{z} \in \mathcal{Z}\), the cost function \(c(\boldsymbol{z}, \cdot)\) in problem~\eqref{eq:dd-csp} is Lipschitz continuous with respect to $\boldsymbol{Y}\in \mathcal{Y}$, i.e.,
\(
|c(\boldsymbol{z}, \bar{\boldsymbol{y}}) - c(\boldsymbol{z}, \boldsymbol{y})| \leq L(\boldsymbol{z}) \| \bar{\boldsymbol{y}} - \boldsymbol{y} \|, \ \forall \boldsymbol{y}, \bar{\boldsymbol{y}} \in \mathcal{Y},
\)
where the Lipschitz modulus \(L(\boldsymbol{z})\) satisfies \(\sup_{\boldsymbol{z} \in \mathcal{Z}} L(\boldsymbol{z}) < +\infty\). 
\end{assumption}

Assumption~\ref{assumption:liptschitz} is satisfied for a broad class of problems, including piecewise-linear costs such as the newsvendor problem presented in Section~\ref{sec:newsvendor} and the two-stage stochastic MILPs with continuous recourse studied in Section~\ref{sec:knn_tractability}.

\begin{assumption}\label{assumption:weak_LLN_and_dominated}
(i) The weak Law of Large Numbers (LLN) holds pointwise for error samples $\{\boldsymbol{\epsilon}^i\}_{i=1}^N$; (ii) for almost every \(\boldsymbol{w} \in \mathcal{W}\), $c(\cdot,\boldsymbol{Y}(\cdot,\boldsymbol{w}))$ is continuous on $\mathcal{Z}$; and (iii) \(c(\cdot, \boldsymbol{Y}(\cdot, \boldsymbol{w}))\) is dominated by an integrable function.
\end{assumption}

Assumption~\ref{assumption:weak_LLN_and_dominated}(i) holds when the error
samples \(\{\boldsymbol{\epsilon}^i\}_{i=1}^N\) are independent and identically distributed (i.i.d.), and more generally for various mixing and stationary
processes \citep{mcleish}. 



\begin{proposition}\label{proposition:saa_convergence}
If \(\mathcal{Z}\) is compact and  Assumption~\ref{assumption:weak_LLN_and_dominated} holds, then for almost every \(\boldsymbol{w} \in \mathcal{W}\), the sequence of sample average functions \(\{   \upsilon_N^\ast(\cdot,\boldsymbol{w}) \}\) defined in~\eqref{eq:fi-er-dd-saa} converges in probability to the true function \(  \upsilon(\cdot, \boldsymbol{w})\) defined in~\eqref{eq:er-dd-saa-true}, uniformly on \(\mathcal{Z}\).
\end{proposition}

Under i.i.d. error samples \(\{\boldsymbol{\epsilon}^i\}_{i=1}^N\), the proof of Proposition~\ref{proposition:saa_convergence} follows from \citet[Theorem~7.48]{shapiro2021lectures}. The proof also extends to non-i.i.d. settings satisfying Assumption~\ref{assumption:weak_LLN_and_dominated}(i) by using pointwise weak LLN results.

\begin{assumption}\label{assumption:estimator}
The regression estimator \(\hat{Q}_N(\boldsymbol{z}, \boldsymbol{w})\) satisfies the following consistency properties: (i) \(\hat{Q}_N(\boldsymbol{z}, \boldsymbol{w}) \xrightarrow{P} Q^\ast(\boldsymbol{z}, \boldsymbol{w})\) uniformly over \((\boldsymbol{z},\boldsymbol{w})\in \mathcal{Z}\times \mathcal{W}\); and (ii) \(
    \frac{1}{N} \sum_{i=1}^N \left\| Q^\ast(\boldsymbol{z}^i, \boldsymbol{w}^i) - \hat{Q}_N(\boldsymbol{z}^i, \boldsymbol{w}^i) \right\| \xrightarrow{P} 0.
    \)
\end{assumption}

We verify that the three nonparametric regression models (i.e., $k$NN, CART, ReLU NNs) satisfy Assumption~\ref{assumption:estimator}(i) under mild conditions in Appendix \ref{subsubsec:ommitted_proof}. Assumption~\ref{assumption:estimator}(ii) is a direct result of Assumption~\ref{assumption:estimator}(i).



\begin{theorem}\label{lem:generic-er-dd-saa}
Suppose Assumptions~\ref{assumption:liptschitz}–\ref{assumption:estimator} hold. Then, for almost every \(\boldsymbol{w} \in \mathcal{W}\), the following results hold for the ER-DD-SAA problem~\eqref{eq:er-DD-SAA}: (i) \(\hat{\psi}_{N,\text{ER}} (\boldsymbol{w}) \xrightarrow{P} \psi ^\ast(\boldsymbol{w})\); (ii) \(\mathbb{D}(\hat{\mathcal{F}}_{N,\mathrm{ER}}(\boldsymbol{w}), \mathcal{F}^\ast(\boldsymbol{w})) \xrightarrow{P} 0\); and (iii) \(\sup_{\boldsymbol{z} \in \hat{\mathcal{F}}_{N,\mathrm{ER}}(\boldsymbol{w})}   \upsilon (\boldsymbol{z}, \boldsymbol{w}) \xrightarrow{P} \psi ^\ast(\boldsymbol{w})\).
\end{theorem}

The proof of Theorem \ref{lem:generic-er-dd-saa} mainly follows from Theorem 1 in \cite{sun2026contextual} and we omit it here. The above theorem states that as the sample size $N$ increases, the optimal solution for the ER-DD-SAA problem~\eqref{eq:er-DD-SAA} converges in probability to that of the true problem~\eqref{eq:dd-csp}, and the objective value of the true problem~\eqref{eq:dd-csp} with the optimal solution of the ER-DD-SAA problem~\eqref{eq:er-DD-SAA} converges to the true optimal objective value in probability.

\section{Numerical Results}\label{sec:experiments}
In this section, we consider two representative problems: a newsvendor problem with continuous pricing decision in Section~\ref{sec:newsvendor} and a two-stage facility location problem with binary first-stage decision in
Section~\ref{sec:facility_location_problem}.  Before solving the optimization problem, we tune the hyperparameters in the nonparametric regression models (e.g., $k$ in $k$NN, tree structure in CART, and network architecture in ReLU NNs). The tuning procedures are described in Appendix~\ref{sec:parameters_tuning}. All
reported results are averaged over five independent runs generated using
different random seeds. All regression models are implemented using
\texttt{scikit-learn} version 1.2.1, and all optimization problems are solved
using Gurobi Optimizer 12.0.3. The entire pipeline is executed in Python 3.10.13
on the high-performance computing resources of the Ohio Supercomputer Center
\citep{osc1987}. All approaches use an optimality gap tolerance of \(10^{-4}\) and a three-hour time limit.

\subsection{Newsvendor Problem with Pricing}\label{sec:newsvendor}
We consider the following newsvendor problem with pricing
\begin{align}\label{eq:newsvendor_cost}
 \min_{p, q\in\mathcal{Z}} \mathbb{E}_Y[f q - pY + h(q-Y)^+ + b(Y-q)^+| P=p, W=w],
\end{align}
where the DM decides the selling price \( p \in [0,20]\) and order quantity \( q \in \mathbb{Z}_{+}\) to minimize the expected total cost. The demand \( Y \) depends on two factors: the pricing decision \( p \), and a covariate \( w \) (e.g., local temperature), which is observed before solving the optimization problem. Following \cite{bertsimas2020predictive}, the ground-truth relationship between the uncertain demand $Y$, pricing decision $p$, and the covariate $w$ is assumed to have the following form: 
\begin{equation}\label{eq:ground_truth}
    Y = \frac{500}{1+e^{p-10}} \max \{0, 10 + 2(w-60)\} + \epsilon
\end{equation}
where \( \epsilon \sim N(0, 3^2) \) is a normally distributed noise term with standard deviation \( \sigma = 3 \). 
The unit procurement cost is \(f=4\). To protect against shortage, the retailer must place an emergency order at a higher unit cost \(b=8>f\) if the realized demand exceeds the initial inventory level \(q\). Any leftover inventory incurs a unit holding cost of \(h=1\). 
We refer readers to Appendix \ref{subsec:newsvendor_full_formulation} for the full formulations of ER-DD-SAA with the three nonparametric regression models in this setting. 

We generate a synthetic dataset \(\mathcal{D}_N = \{ (p^i, w^i, y^i) \}_{i=1}^{N}\) with \(N = 10000\) observations. To generate each data point, the price \( p^i \) is sampled from a log-normal distribution \( \log(p) \sim N(2.3, 0.25^2) \), the temperature \( w^i \) is sampled from \( N(72, 5^2) \), and the corresponding demand $y^i$ is constructed from the ground truth demand function \eqref{eq:ground_truth}. To evaluate model performance, we
compute the out-of-sample (OOS) cost of the obtained optimal solution. For each
test instance, we generate $1000$ independent OOS demand realizations from the
ground-truth demand model evaluated at the obtained solution, with independently
sampled noise terms added to each realization.

\paragraph{Comparison between different nonparametric regression models.}
We compare \(k\)NN, CART, and ReLU NNs with a linear regression benchmark:
\(
Y=\tau_p p+\tau_w w+\rho + \epsilon,
\)
estimated by ordinary least squares on the same training data \(\mathcal{D}_N\). For \(k\)NN and CART, the products involving pricing decision \(p\) and binary variables \(\boldsymbol t\) and \(\boldsymbol R\) can be linearized exactly using McCormick envelopes. The resulting formulations are therefore MILPs. In contrast, the linear regression benchmark and ReLU NNs involve products between continuous pricing decision \(p\) and continuous prediction variables \(\hat{y}\), leading to nonconvex formulations. Table~\ref{tab:newsvendor_models_comparison} reports the average training time (Train), optimization time (Opt.), both measured in seconds, and OOS cost. 

Table~\ref{tab:newsvendor_models_comparison} shows that CART yields the shortest optimization time across all sample sizes, and it achieves better OOS performance than linear regression and $k$NN as the sample size increases. The \(k\)NN-based approach generally improves OOS performance relative to linear regression and CART for smaller sample sizes, but it requires substantially longer optimization time. For \(N=1200\), only \(3/5\) instances using \(k\)NN are solved to optimality. ReLU NNs achieve the best OOS performance with manageable optimization time across all tested sample sizes. Comparisons of our proposed formulations with other benchmark formulations are presented in Appendix~\ref{sec:knn_comparison}--~\ref{sec:nn_comparison_newsvendor}.

\begin{table}[htbp]
\centering
\caption{Performance comparison across different regression models}
\label{tab:newsvendor_models_comparison}
\resizebox{\textwidth}{!}{%
\begin{tabular}{c|ccc|ccc|ccc|ccc}
\hline
\multirow{2}{*}{$N$}
& \multicolumn{3}{c|}{Linear Regression}
& \multicolumn{3}{c|}{$k$NN~\eqref{eq:single_level_binary}}
& \multicolumn{3}{c|}{CART~\eqref{eq:er_dd_saa_cart}}
& \multicolumn{3}{c}{ReLU NNs~\eqref{eq:nn_erddsaa}} \\
\cline{2-4} \cline{5-7} \cline{8-10} \cline{11-13}
& Train & Opt. & OOS Cost
& Train & Opt. & OOS Cost
& Train & Opt. & OOS Cost
& Train & Opt. & OOS Cost \\ \hline
600  & 0.02 & 26.90  & -49,498.94 & 6.54 & 553.18   & -53,383.53 & \textbf{0.00} & \textbf{6.43} & -49,145.96 & 25.84 & 11.48  & \textbf{-54,957.83} \\
800  & 0.02 & 44.56  & -49,373.91 & 6.54 & 1,195.71 & -52,814.03 & \textbf{0.00} & \textbf{6.84} & -51,124.38 & 35.00 & 26.66  & \textbf{-55,070.71} \\
1000 & 0.02 & 75.50  & -49,350.27 & 6.73 & 4,287.74 & -53,725.28 & \textbf{0.00} & \textbf{6.97} & -51,500.96 & 39.85 & 48.15  & \textbf{-55,136.74} \\
1200 & 0.01 & 100.32 & -49,188.07 & 6.37 & 7,646.40 & -50,876.30 & \textbf{0.00} & \textbf{7.34} & -52,216.36 & 44.69 & 103.76 & \textbf{-55,132.69} \\ \hline
\end{tabular}%
}
\end{table}

\subsection{Two-stage facility location problem}\label{sec:facility_location_problem}
We next consider a two-stage facility location problem with decision-dependent demand. Let $\Gamma_1$ denote the set of candidate facilities and $\Gamma_2$ the set of customer sites. Here, we consider a problem with \(|\Gamma_1|=20, \, |\Gamma_2|=10\). Let $z_i \in \{0, 1\}$ be the binary variable indicating whether facility $i \in \Gamma_1$ is chosen, with $f_i$ representing its fixed opening cost. The two-stage facility location problem is:
\begin{align}\label{eq:two_stage_facility_location_problem}
    \min_{\boldsymbol{z} \in \{0,1\}^{|\Gamma_1|}} \quad
    & \sum_{i\in \Gamma_1} f_i z_i
    + \mathbb{E}\!\left[
        h\big(\boldsymbol{z},\boldsymbol{Y}(\boldsymbol{z},\boldsymbol{w})\big)
        \mid \boldsymbol{Z}=\boldsymbol{z}, \boldsymbol{W}=\boldsymbol{w}
    \right] 
\end{align}
where the second-stage recourse function $h(\boldsymbol{z},\boldsymbol{Y}(\boldsymbol{z}, \boldsymbol{w}))$ and the full formulations of ER-DD-SAA under this setting are presented in Appendix \ref{subsec:facility_formulation}.

We assume that the demand $\boldsymbol{Y}$ depends on the number of open facilities $\sum_{j} z_j$ and gas-price covariate $e$, which follows the nonlinear ground-truth function: 
\begin{equation}
Y = \chi_1 + \chi_2 \left[ \log(\chi_3+1) \frac{1 - \exp(-\sum_{j} z_j/7)}{1 - \exp(-\chi_3/7)} \right] \max\{0, 5e - 10\} + \epsilon.
\end{equation}
This function ensures that demand increases with the number of open facilities $\sum_{j} z_j$, but the increase becomes slower as more facilities are opened. Demand also increases with the gas price $e$ up to a threshold.  The detailed parameter setup is described in Appendix \ref{subsec:facility_formulation}.
We generate synthetic dataset \(\mathcal{D}_N=\{(\sum_{j} z_j^i, e^i, y^i)\}_{i=1}^N\) with a size of \(N=10000\). For each sample, the gas price and the number of open facilities are sampled according to $e \sim N(4,1)$ and $ \sum_{j} z_j \sim U(0, 30)$, respectively, and the random noise is generated as $\epsilon \sim N(0,2^2)$. 


\paragraph{Comparisons between different approaches for ER-DD-SAA with \(k\)NN.}
Table~\ref{tab:distance_algorithm_comparison} compares the performance of Gurobi, vanilla Bender's decomposition, and BD-CG (Algorithm~\ref{alg:row_generation_knn}) under formulation~\eqref{eq:distance_compare_tie}. In Step~5 of Algorithm~\ref{alg:row_generation_knn}, we also add critical observations whose demand values exceed those in the current true \(k\)NN set to set \(J\). Here, ``TL'' indicates that 
the time limit was reached. From Table~\ref{tab:distance_algorithm_comparison}, Gurobi and vanilla Benders decomposition both fail to find a feasible solution within the time limit for \(N\ge 1000\), whereas BD-CG solves all instances within the time limit. This demonstrates the efficiency of our BD-CG algorithm in solving formulation \eqref{eq:distance_compare_tie}.

\begin{table}[htbp]
\centering
\caption{Computational comparison for ER-DD-SAA with \(k\)NN under formulation~\eqref{eq:distance_compare_tie} when $k$=1.}
\label{tab:distance_algorithm_comparison}
\resizebox{\textwidth}{!}{%
\begin{tabular}{c|rrr|rrr|rrr}
\hline
\multirow{2}{*}{$N$} & \multicolumn{3}{c|}{Gurobi} & \multicolumn{3}{c|}{Vanilla BD} & \multicolumn{3}{c}{BD-CG (Algorithm \ref{alg:row_generation_knn})} \\
\cline{2-4} \cline{5-7} \cline{8-10}
& \multicolumn{1}{c}{IS Cost} & \multicolumn{1}{c}{Gap} & \multicolumn{1}{c|}{Time (s)} & \multicolumn{1}{c}{IS Cost} & \multicolumn{1}{c}{Gap} & \multicolumn{1}{c|}{Time (s)} & \multicolumn{1}{c}{IS Cost} & \multicolumn{1}{c}{Gap} & \multicolumn{1}{c}{Time (s)} \\
\hline
500  & -2,017,506.78 & 0.66 & 8,511.47  & -2,016,845.65 & 0.09 & 10,263.96 & -2,017,506.78 & 0.00 & 357.85   \\
1000 & - & - & TL & - & - & TL & -2,010,574.63 & 0.01 & 1,568.19 \\
1500 & - & - & TL & - & - & TL & -2,002,573.35 & 0.00 & 3,881.70 \\
2000 & - & - & TL & - & - & TL & -1,992,462.15 & 0.00 & 9,026.81 \\ \hline
\end{tabular}%
}
\end{table}

We further compare our proposed formulation~\eqref{eq:distance_compare_tie} solved by BD-CG (Algorithm \ref{alg:row_generation_knn}), bilevel formulation \eqref{eq:single_level_binary} solved by Gurobi and a benchmark proposed in \cite{liu2023solving} solved by Gurobi. In particular, in \cite{liu2023solving}, constraints \eqref{eq:random_select_knn_strict1} and \eqref{eq:random_select_knn_strict3} are replaced by
\begin{equation}\label{eq:zhihai}
    s_i - s_j \le M(t_j - t_i + 1), 
    \quad \forall i,j \in [N],\ i \ne j.
\end{equation} Table~\ref{tab:benders_bdcg_comparison} reports the average IS cost, optimality gap, and runtime.  From Table~\ref{tab:benders_bdcg_comparison}, BD-CG also consistently outperforms the bilevel formulation \eqref{eq:single_level_binary} and the benchmark~\eqref{eq:zhihai} solved by Gurobi directly, which often reaches the time limit with a large gap or without finding a feasible solution. 

\begin{table}[htbp]
\centering
\caption{Computational comparison of bilevel formulation \eqref{eq:single_level_binary}, benchmark \eqref{eq:zhihai}, and BD-CG (Algorithm \ref{alg:row_generation_knn}).}
\label{tab:benders_bdcg_comparison}
\resizebox{\textwidth}{!}{%
\begin{tabular}{l|rrr|rrr|rrr}
\hline
\multirow{2}{*}{$N$} & \multicolumn{3}{c|}{Bilevel~\eqref{eq:single_level_binary}}
& \multicolumn{3}{c|}{Benchmark~\eqref{eq:zhihai}}
& \multicolumn{3}{c}{BD-CG (Algorithm \ref{alg:row_generation_knn})} \\
\cline{2-4} \cline{5-7} \cline{8-10}
& \multicolumn{1}{c}{IS Cost} & \multicolumn{1}{c}{Gap} & \multicolumn{1}{c|}{Time (s)}
& \multicolumn{1}{c}{IS Cost} & \multicolumn{1}{c}{Gap} & \multicolumn{1}{c|}{Time (s)}
& \multicolumn{1}{c}{IS Cost} & \multicolumn{1}{c}{Gap} & \multicolumn{1}{c}{Time (s)} \\ \hline
500  & -2,017,506.78 & 0.00  & 794.40   & -2,017,506.78 & 0.00   & 3,807.37 & -2,017,506.78 & 0.00 & 357.85   \\
1000 & -2,010,574.63 & 0.00  & 4,792.86 & -1,507,570.12 & 101.60 & TL       & -2,010,574.63 & 0.01 & 1,568.19 \\
1500 & -1,994,364.92 & 9.35  & 9,002.10 & --            & --     & TL       & -2,002,573.35 & 0.00 & 3,881.70 \\
2000 & -1,978,302.08 & 23.48 & TL       & --            & --     & TL       & -1,992,462.15 & 0.00 & 9,026.81 \\ \hline
\end{tabular}%
}
\end{table}


\paragraph{Comparison between different nonparametric regression models.}
Table~\ref{tab:facility_models_comparison} compares linear regression, \(k\)NN, CART, and ReLU NNs within the ER-DD-SAA framework. 
Table~\ref{tab:facility_models_comparison} shows that ER-DD-SAA with linear regression is the fastest but yields the weakest OOS performance. CART provides a favorable balance between solution quality and computational time. The \(k\)NN model improves upon linear regression but becomes increasingly expensive as \(N\) grows. ReLU NNs achieve the best OOS performance for all tested sample sizes and are generally faster to solve than \(k\)NN. Additional numerical results on different CART formulations can be found in Appendix~\ref{sec:facility_cart_comparison}.

\begin{table}[htbp]
\centering
\caption{Computational comparison between different nonparametric regression models.}
\label{tab:facility_models_comparison}
\resizebox{\textwidth}{!}{%
\begin{tabular}{c|rrr|rrr|rrr|rrr}
\hline
\multirow{2}{*}{$N$}
& \multicolumn{3}{c|}{Linear Regression}
& \multicolumn{3}{c|}{$k$NN (BD-CG, $k=1$)}
& \multicolumn{3}{c|}{CART~\eqref{eq:er_dd_saa_cart}}
& \multicolumn{3}{c}{ReLU NNs~\eqref{eq:nn_erddsaa}} \\
\cline{2-4} \cline{5-7} \cline{8-10} \cline{11-13}
& \multicolumn{1}{c}{Train} & \multicolumn{1}{c}{Opt.} & \multicolumn{1}{c|}{OOS Cost}
& \multicolumn{1}{c}{Train} & \multicolumn{1}{c}{Opt.} & \multicolumn{1}{c|}{OOS Cost}
& \multicolumn{1}{c}{Train} & \multicolumn{1}{c}{Opt.} & \multicolumn{1}{c|}{OOS Cost}
& \multicolumn{1}{c}{Train} & \multicolumn{1}{c}{Opt.} & \multicolumn{1}{c}{OOS Cost} \\ \hline
500  & \textbf{0.00} & 2.81           & -1,815,424.99 & 5.98 & 357.85   & -1,849,296.57 & 9.47  & 20.05 & -1,852,206.52 & 20.62 & \textbf{16.44} & \textbf{-1,855,540.72} \\
1000 & \textbf{0.00} & \textbf{5.88}  & -1,815,424.99 & 6.97 & 1,568.19 & -1,850,493.19 & 9.67  & 49.21 & -1,853,490.67 & 21.74 & 525.06         & \textbf{-1,855,689.25} \\
1500 & \textbf{0.06} & \textbf{9.17}  & -1,815,424.99 & 8.29 & 3,881.70 & -1,851,152.89 & 9.69  & 26.19 & -1,852,041.71 & 27.09 & 51.84          & \textbf{-1,856,184.15} \\
2000 & \textbf{0.00} & \textbf{12.50} & -1,815,424.99 & 5.86 & 9,026.81 & -1,854,126.37 & 14.30 & 90.62 & -1,855,618.67 & 44.38 & 94.94          & \textbf{-1,856,029.75} \\ \hline
\end{tabular}%
}
\end{table}





\section{Conclusions}
This paper studied contextual stochastic programming with decision-dependent uncertainty, where the uncertainty depends jointly on the decision variables and contextual information. We proposed a unified ER-DD-SAA framework that integrates learned nonparametric regression models, including \(k\)NN, CART, and ReLU NNs, into the downstream stochastic optimization problem. Under suitable structural conditions, we derived exact MILP/MINLP reformulations and established consistency and asymptotic optimality of the proposed ER-DD-SAA framework with three nonparametric regression models. For ER-DD-SAA with \(k\)NN, we developed both a pairwise distance comparison formulation and a bilevel formulation. To improve computational tractability, we further designed a tailored decomposition algorithm that combines Bender's decomposition with constraint generation. Numerical experiments on newsvendor and two-stage facility location problems demonstrate the effectiveness of the proposed methods in terms of optimization performance and highlight the trade-off among predictive flexibility, optimization quality, and computational efficiency when embedding different regression models into ER-DD-SAA. Future work includes improving the computational efficiency of \(k\)NN-based formulations under nonlinear distance metrics, extending the framework to more complex regression models, and incorporating multiple predictors and richer covariate structures.


  


\bibliography{mybib}

\newpage
\appendix
\section{Omitted Formulations}\label{sec:omitted_formulations}
\subsection{ER-DD-SAA with \(k\)NN}
\label{appendix:ranking}
\textbf{Distance Ranking Formulation.}
In this formulation, we introduce ranks \(p=1,\ldots,N\) to explicitly represent the ranking of the distances \(s_j, \, \forall j \in [N]\) from smallest to largest. The observations assigned to the first \(k\) ranks are then identified as the \(k\) nearest neighbors. We define binary variables \(o_{i,p}\in\{0,1\}\), such that \(o_{i,p}=1\) if observation \(i\) is assigned to rank \(p\), and \(o_{i,p}=0\) otherwise. Then ER-DD-SAA with \(k\)NN can be formulated as
\allowdisplaybreaks
\begin{subequations}
\begin{align}
\min_{\boldsymbol d, \boldsymbol{o}, \boldsymbol{s}, \boldsymbol{t}, \boldsymbol z}\quad & \frac{1}{N}\sum_{i=1}^N
    c\!\left(
        \boldsymbol z,
        \operatorname{Proj}_{\mathcal Y}\!\left(
            \frac{1}{k}\sum_{j=1}^N t_j y^j + \hat{\epsilon}^i
        \right)
    \right) \\
    \text{s.t.} \quad 
    & \eqref{eq:decision_feasibility}--\eqref{eq:random_select_knn_p_norm} \\
    & \sum_{p=1}^N o_{i,p} = 1, \quad \forall i \in [N], \label{eq:ordering_assign}\\
    & \sum_{i=1}^N o_{i,p} = 1, \quad \forall p \in [N],  \label{eq:ordering_assigN_1} \\
    & \sum_{i=1}^N s_i  o_{i,p}  \leq \sum_{i=1}^N s_i  o_{i,p+1}, \quad \forall p\in [N-1],  \label{eq:ordering_distance_order} \\
    & t_i = \sum_{p=1}^k o_{i, p}, \quad \forall i\in [N], \label{eq:ordering_knn_selection} \\
    & o_{i,p} \in \{0,1\}, \quad \forall i\in [N], p \in [N],
\end{align}
\label{eq:ordering}
\end{subequations}
where constraints~\eqref{eq:ordering_assign} assign each historical observation to exactly one rank, and constraints~\eqref{eq:ordering_assigN_1} ensure that each rank is occupied by exactly one observation. Constraints~\eqref{eq:ordering_distance_order} enforce a non-decreasing ranking of the distances. Constraints~\eqref{eq:ordering_knn_selection} then identify the \(k\) nearest neighbors by determining whether observation \(i\) is assigned to one of the first \(k\) ranks. Let \(\overline{s}_i\) be a valid upper bound on \(s_i\), so that \(s_i \in [0, \overline{s}_i]\). The bilinear terms \(s_i  o_{i,p}\) can be linearized using McCormick envelopes \citep{mccormick1976computability} with a set of constraints
\begin{equation}\label{eq:mccormick}
0\le m_{i,p}\le s_i, \, 
s_i+\bar{s}_i(o_{i,p}-1)\le m_{i,p}\le \bar{s}_io_{i,p},\ \forall i,p\in[N].
\end{equation}

We show that this distance ranking formulation (23) is equivalent to pairwise distance comparison formulation (5) in the next theorem. \begin{theorem}\label{thm:knn_equivalence}
    Model \eqref{eq:distance_compare_tie} and Model \eqref{eq:ordering} are equivalent.
\end{theorem}

\begin{proof}\label{proof:knn}
    \begin{enumerate}
    \item Suppose \((\boldsymbol{s}^\ast, \boldsymbol{o}^\ast, \boldsymbol{t}^\ast)\) is an optimal solution to Model \eqref{eq:ordering}. Below, we construct a feasible solution to Model \eqref{eq:distance_compare} with the same objective function value. We set \(s_i=s_i^\ast,\ t_i = t_i^\ast,\ \forall i \in [N]\) and \( d_{ij}= \sum_{p=1}^N o^\ast_{i,p} \sum_{q=1}^p o^\ast_{j,q}, \ \forall i\in [N],\ j\in[N]\).

Due to constraints~\eqref{eq:ordering_assign} and \eqref{eq:ordering_assigN_1}, \( \sum_{q=1}^p o^\ast_{i,q} \in \{0,1\}, \text{and} \sum_{p=1}^N o^\ast_{j,p} \sum_{q=1}^p o^\ast_{i,q} \in \{0,1\},\) we have \( d_{ij} \in \{0,1\}.\)

From constraint \eqref{eq:ordering_distance_order}, if \(s_i^\ast \leq s_j^\ast\), then for the unique index \(p^\ast\) such that \(o^\ast_{i,p^\ast} = 1\), we have \(\sum_{q=1}^{p^\ast} o^\ast_{j,q} = 0\). Together with constraint~\eqref{eq:ordering_assign} and \eqref{eq:ordering_assigN_1}, for $p^\prime \neq p^\ast$, we have $o_{i,p^\prime}=0$, this implies that
\[
\sum_{p=1}^N o^\ast_{i,p} \sum_{q=1}^p o^\ast_{j,q} = 0 \quad \text{if } s_i^\ast \leq s_j^\ast.
\]
Conversely, if \(s_i^\ast \geq s_j^\ast\), then from constraint \eqref{eq:ordering_distance_order}, for the unique \(p^\ast\) satisfying \(o^\ast_{i,p^\ast} = 1\), we have \(\sum_{q=1}^{p^\ast} o^\ast_{j,q} = 1\). Similarly, combined with constraints~\eqref{eq:ordering_assign} and \eqref{eq:ordering_assigN_1}, it follows that
\[
\sum_{p=1}^N o^\ast_{i,p} \sum_{q=1}^p o^\ast_{j,q} = 1 \quad \text{if } s_i^\ast \geq s_j^\ast.
\]
Therefore, we have
\[
d_{ij} = \sum_{p=1}^N o^\ast_{i,p} \sum_{q=1}^p o^\ast_{j,q}
=
\begin{cases}
0, & \text{if } s_i^\ast \leq s_j^\ast, \\
1, & \text{if } s_i^\ast \geq s_j^\ast.
\end{cases}
\]
Since $s_i = s_i^\ast$, it ensures that \(d_{ij}\) satisfies constraints \eqref{eq:random_select_knn_strict1}.

From the definition of the ordering variables $\boldsymbol{o}$, if $s_i $ is among the $k$ nearest neighbors, i.e., $s_i $ belongs to the $k$ smallest distances, then there exists $p^\ast \leq k$ such that $o^\ast_{i,p^\ast} = 1$, and thus 
\(
t_i^\ast = \sum_{p=1}^k o^\ast_{i,p} = 1.
\)
Together with constraints \eqref{eq:ordering_assign} and \eqref{eq:ordering_assigN_1}, we have
\(
\sum_{j=1}^N \sum_{q=1}^{p^\ast} o^\ast_{j,q} \leq k,
\)
which implies
\(
\sum_{j=1}^N d_{ij} = \sum_{j=1}^N o^\ast_{i,p^\ast}\sum_{q=1}^{p^\ast} o^\ast_{j,q} \leq k.
\)

Conversely, if $s_i $ is not among the $k$ nearest neighbors, then there exists $p^\prime > k$ such that $o^\ast_{i,p^\prime} = 1$, and thus
\(
t_i^\ast = \sum_{p=1}^k o^\ast_{i,p} = 0.
\)
As a result, 
\(
\sum_{j=1}^N \sum_{q=1}^{p^\prime} o^\ast_{j,q} > k,
\)
\(
\sum_{j=1}^N d_{ij} = \sum_{j=1}^N o^\ast_{i,p^\prime}\sum_{q=1}^{p^\ast} o^\ast_{j,q} > k.
\)

Since $t_i=t^\ast_i$, $d_{ij}$ satisfies constraints \eqref{eq:random_select_knn_strict3}. From constraints \eqref{eq:ordering_knn_selection},  \eqref{eq:ordering_assign} and \eqref{eq:ordering_assigN_1}, we have $\sum_{i=1}^N t_i = \sum_{i=1}^N t_i^\ast = \sum_{i=1}^N \sum_{j=1}^k o^\ast_{ij}=k$, which implies that $t_i$ satisfies constraint \eqref{eq:random_select_knn_sum_t}.

\item Conversely, suppose \( (t^\ast, s^\ast, d^\ast)\) is an optimal solution to Model \eqref{eq:distance_compare}. We can construct a feasible solution to Model \eqref{eq:ordering} that attains the same objective value. Let $t_i=t^\ast_i, \bar{s} = s_i^\ast$.

For each data point \(i\), consider the sum
\(\sum_{j=1}^N d^\ast_{ij}\), and sort these sums in non-decreasing order. We define the ranking index \(q^\ast_i\) for each data point \(i\) as the position of \(\sum_{j=1}^N d^\ast_{ij}\) in the ordered sequence. Then, we let
\(
\bar{o}_{i,p} = \mathbb{I}\{q^\ast_i = p\}.
\)
With this ranking index, each data point is assigned to exactly one position among \(N\), and \(\bar{o}_{i,p}\) satisfies constraints \eqref{eq:ordering_assign} and \eqref{eq:ordering_assigN_1}.

From constraints \eqref{eq:random_select_knn_strict1}, we have
\(
s_i  \leq s_j  \Longleftrightarrow \sum_{j=1}^N d^\ast_{ij} \leq \sum_{i=1}^N d^\ast_{ji}.
\)
Thus, we have
\(
\sum_{i=1}^N s_i  \bar{o}_{i,p}
= \sum_{i=1}^N s_i  \cdot \mathbb{I}\{q^\ast_i = p\}
= \{s_i  \mid q^\ast_i = p\}
\leq \{s_i  \mid q^\ast_i = p+1\}
= \sum_{i=1}^N s_i  \bar{o}_{i,p+1}.
\)
This implies that \(\bar{o}_{i,p}\) satisfies constraints \eqref{eq:ordering_distance_order}.

From constraints \eqref{eq:random_select_knn_strict3}, and \eqref{eq:random_select_knn_sum_t}, we have $t_i^\ast = \mathbb{I}\{q_i^\ast \leq k \}$. From constraints~\eqref{eq:ordering_assign}, we have
\(
t_i = \sum_{p=1}^k o_{i, p} = \sum_{p=1}^k \mathbb{I}\{q_i^\ast=p\} = \mathbb{I}\{q_i^\ast \leq k \} = t_i^\ast,
\)
which implies $t_i$ satisfies constraint ~\eqref{eq:ordering_knn_selection}.
\end{enumerate}
Since both formulations \eqref{eq:distance_compare_tie} and \eqref{eq:ordering} have the same objective function, this concludes the proof. 
\end{proof}
Although these two formulations are equivalent, during our initial testing, we found that Model \eqref{eq:ordering} resulted in much longer solution times. Therefore, we omit Model \eqref{eq:ordering} in the comparison.

\subsection{ER-DD-SAA with CART}\label{sec:er_dd_saa_cart_omitted}
\textbf{Special Case: Single-Dimensional Decision.} When the decision variable is one-dimensional, i.e., \(d_z=1\), we propose an alternative formulation tailored to this special case. In this case, each region is completely determined by its upper-bound threshold. Let
\(\{b_z^{(1)},\dots,b_z^{(N^{\boldsymbol{w}})}\}\)
denote these upper-bound values sorted in nondecreasing order:
$b_z^{(1)} \le b_z^{(2)} \le \cdots \le b_z^{(N^{\boldsymbol{w}})}$. Given a target point \(\boldsymbol{z}\), the goal is to identify the smallest index \(r\in [N^{\boldsymbol{w}}]\) such that \(z\le b_z^{(r)}\), which corresponds to the first region in the ordered sequence that contains \(\boldsymbol{z}\). 

To model such logic, we introduce binary variables \(\bar u_r\in\{0,1\}\), such that \(\bar u_r=1\) if and only if \(z\le b_z^{(r)}\). We continue to use binary variables \(R_r\in\{0,1\}\) to indicate whether region $r$ is selected, where \(R_r=1\) if and only if \(r\) is the smallest index such that \(\bar u_r=1\). The resulting formulation is
\allowdisplaybreaks
 \begin{subequations}
\begin{align}\min_{\boldsymbol{l},\boldsymbol{u},\boldsymbol{R}, \boldsymbol{z}} \quad & \frac{1}{N}\sum_{i=1}^N
    c\!\left(
        \boldsymbol z,
        \operatorname{Proj}_{\mathcal Y}\!\left(
            \sum_{r=1}^{N^{\boldsymbol{w}}} R_r \hat{y}^r + \hat{\epsilon}^i
        \right)
    \right) \\
    \text{s.t.} \quad 
    & \delta - M \bar{u}_r \le z - b_z^{(r)} \leq M (1 - \bar{u}_r),\quad \forall r \in \mathcal{N}^{\boldsymbol w},  \label{eq:upper_bound_ordered_rule} \\
    & R_r \leq \bar{u}_r, \quad  \forall r\in \mathcal{N}^w, \label{eq:satisfy_check_ordered_rule} \\
    & R_r \leq \bar{u}_r - \bar{u}_{r-1}, \quad  \forall r\in \mathcal{N}^w \setminus \{1\} \label{eq:smallest_r_ordered_rule} \\
    & \bar{u}_r, R_r \in \{0, 1\}, \quad \forall r\in \mathcal{N}^w,\label{eq:ordered_binary} \\
    & \eqref{eq:cart_decision_feasiblity}--\eqref{eq:unique-box}.
\end{align}
\label{eq:special_cart}
\end{subequations}
Here, constraints~\eqref{eq:upper_bound_ordered_rule} enforce that \(\bar u_r=1\) if and only if \(z\le b_z^{(r)}\). Constraints~\eqref{eq:satisfy_check_ordered_rule} impose that region \(r\) can be selected only if $\bar{u}_r=1$, i.e., \(z \le b_z^{(r)}\), where we use a sufficiently large constant $M_r= b_z^{(r)}- b_z^{(r-1)}$. Constraints~\eqref{eq:smallest_r_ordered_rule} ensure that region \(r\) can be selected only if \(r\) is the first index for which \(\bar u_r=1\), i.e., only if \(z\le b_{(r)}\) and \(z>b_{(r-1)}\).

Model~\eqref{eq:er_dd_saa_cart} and Model~\eqref{eq:special_cart} encode the same decision logic in the one-dimensional setting. While Model~\eqref{eq:special_cart} saves $N^{\boldsymbol{w}}$ binary variables by omitting explicit boundary definitions, it requires additional $N^{\boldsymbol{w}}$ constraints to identify the active leaf node. The following theorem establishes the equivalence between the two formulations.

\begin{theorem}\label{thm:cart_equivalence}
    Model \eqref{eq:er_dd_saa_cart} and Model \eqref{eq:special_cart} are equivalent when \(d_z=1\). 
\end{theorem}

\begin{proof}\label{sec:proof_single_equivalence}
   Let $a_z^r = b_z^{(r-1)}$ for $r = 2, \ldots, N^{\boldsymbol{w}}$. Since the
leftmost region has lower bound $-\infty$ from training, we set
$a_z^1 < \inf_{z \in \mathcal{Z}} z$ so that $l_1 = 1$ for every feasible $z$; this is without loss of generality and does not alter the CART partition.  Define $b_z^r = b_z^{(r)}$ for $r \in \mathcal{N}^{\boldsymbol{w}}$. 

    \begin{enumerate}
        \item Suppose \( (\boldsymbol{l}^\ast, \boldsymbol{u}^\ast, \boldsymbol{R}^\ast) \) is an optimal solution to Model \eqref{eq:er_dd_saa_cart}. We can then construct a feasible solution \( (\bar{\boldsymbol{u}}^\prime, \boldsymbol{R}^\prime) \) to Model \eqref{eq:special_cart} by setting \(
\bar{u}^\prime_{r-1} = 1 - l_r^\ast, \ \bar{u}^\prime_{r} = u_r^\ast, \  R^\prime_r = R_r^\ast.
\)
By construction, \( \bar{\boldsymbol{u}}^\prime \) and \( \boldsymbol{R}^\prime \) satisfy constraints \eqref{eq:unique-box} and \eqref{eq:cart_binary_decision}.  From constraints \eqref{eq:ind_region_upperbound1}, it follows that \( \bar{u}^\prime_r \) satisfies constraints \eqref{eq:upper_bound_ordered_rule}. Moreover, from constraint \eqref{eq:box-activation}, if \( R_r^\ast = 1 \), then \( l_r^\ast = 1 \) and \( u_r^\ast = 1 \). Therefore, \( \bar{u}^\prime_{r-1} = 0 \), \( \bar{u}^\prime_r = 1 \), and thus \( \bar{u}^\prime_r \geq R_r^\ast = R^\prime_r \),  \( \bar{u}^\prime_r - \bar{u}^\prime_{r-1} = 1 \geq R_r^\ast = R^\prime_r \). If \( R_r^\ast = 0 \), then \( ( l_r^\ast, u_r^\ast ) = (0, 1) \) or \( (1, 0) \). Therefore, \( (\bar{u}^\prime_{r-1}, \bar{u}^\prime_r) = (1, 1) \) or \( (0, 0) \), which implies \( \bar{u}^\prime_r \geq R_r^\ast = R^\prime_r \), \( \bar{u}^\prime_r - \bar{u}^\prime_{r-1} = 0 \geq R_r^\ast = R^\prime_r \). In both cases, \( \bar{\boldsymbol{u}}^\prime \) and \( \boldsymbol{R}^\prime \) satisfy constraints~\eqref{eq:satisfy_check_ordered_rule} and \eqref{eq:smallest_r_ordered_rule}.

\item Conversely, suppose \((\boldsymbol{\bar{u}}^\ast, \boldsymbol{R}^\ast)\) is an optimal solution to Model \eqref{eq:special_cart}, we can construct a feasible solution \( (\boldsymbol{l}^\prime, \boldsymbol{u}^\prime, \boldsymbol{R}^\prime)\) to Model \eqref{eq:er_dd_saa_cart} by setting $\boldsymbol{R}^\prime = \boldsymbol{R}^\ast, \boldsymbol{u}^\prime=\boldsymbol{\bar{u}}^\ast$ , $l^\prime_r=1 - \bar{u}^\ast_{r-1}, \forall r=2,\ldots, N^{\boldsymbol{w}}$ and $l^\prime_1=1$. By construction, \( (\boldsymbol{l}^\prime, \boldsymbol{u}^\prime, \boldsymbol{R}^\prime)\) satisfies constraints~\eqref{eq:cart_binary_decision} and \eqref{eq:unique-box}. From constraints~\eqref{eq:upper_bound_ordered_rule}, \(\boldsymbol{l}^\prime, \boldsymbol{u}^\prime\) satify constraints~\eqref{eq:ind_region_lowerbound1}--\eqref{eq:ind_region_upperbound1}. From constraint~\eqref{eq:cart_binary_decision}, $R_r=1$ if and only if $\bar{u}^\ast_r=1,$ and $\bar{u}^\ast_{r-1}=0$, then by construction $l_r^\prime = 1- \bar{u}^\ast_r=1$ and $u^\prime_r=\bar{u}^\ast_r=1$, which implies \(\boldsymbol{l}^\prime, \boldsymbol{u}^\prime\) satisfy constraint~\eqref{eq:box-activation}.

    \end{enumerate}    
    This completes the proof.
\end{proof}

\section{Omitted Proofs}\label{subsubsec:ommitted_proof}

\noindent\textbf{THEOREM~\ref{thm:alg_convergence}.}
\textit{
Suppose Assumptions~\ref{ass:complete_recourse} and~\ref{ass:expensive_recourse} hold,
\(\mathcal Z\) is nonempty and compact, and the
RMP~\eqref{eq:master_problem} is solved to
global optimality at each iteration. The
proposed BD-CG algorithm converges to a globally optimal solution in finitely
many iterations.
}
 
\begin{proof}\label{proof:algorithm_convergence}
We first show that, under the exact termination criterion, the algorithm
returns a globally optimal solution upon termination. We then show that
termination must occur within finitely many iterations.
\begin{enumerate}
    \item \textbf{Global optimality upon termination.}
    \begin{enumerate}
        \item We first introduce an
        equivalent full formulation for ER-DD-SAA problem~\eqref{eq:two_stage_problem_saa} as follows:
        \begin{subequations}\label{eq:full_problem} 
\begin{align} 
\min_{\boldsymbol z, \boldsymbol{d}, \boldsymbol{t}, \theta} \quad & \boldsymbol{f}^{\top}\boldsymbol{z} + \theta \\ 
\text{s.t.} \quad & \eqref{eq:decision_feasibility}--\eqref{eq:random_select_knn_sum_t}, \eqref{eq:knn_binary}, \\ 
& d_{ij} + d_{ji}=1, \quad \forall i, j \in [N], \, i <j, \\ 
& M_1^{ij}(d_{ij}-1) \le s_i - s_j \le M_1^{ij} d_{ij}, \quad \forall i, j \in [N], \, i <j, \\ 
& k-M_2 t_i \le \sum_{j \in [N]\setminus\{i\}} d_{ij} + 1  \le k + M_2(1-t_i), \quad \forall i \in [N], \label{eq:alg_constraint3}\\ 
& d_{ij} \in \{0,1\}, \quad \forall i, j \in [N], \\
& \theta \ge \frac{1}{N}\sum_{i=1}^{N}
            h\!\left(\boldsymbol z,\ \frac{1}{k}\sum_{j=1}^{N}t_j\boldsymbol y^{j}
            +\hat{\boldsymbol\epsilon}_i\right). \label{eq:full_second_approximation}
\end{align} 
\end{subequations}
 
        \item At iteration \(m\), the RMP~\eqref{eq:master_problem} is a
        relaxation of the full problem~\eqref{eq:full_problem}. This follows
        from two sources. First, the RMP includes only a subset of the
        constraints required to exactly characterize the \(k\)NN selection.
        Second, the RMP contains only the Bender's optimality cuts generated
        up to iteration \(m\), rather than all cuts associated with the
        extreme points of the dual recourse polyhedron. For any generated
        dual extreme point \(\boldsymbol\pi^{\ell}\), the corresponding
        Bender's cut \eqref{eq:dual_cut} gives a lower approximation of the recourse function
        by weak duality: \(
            h\!\left(\boldsymbol z,\ \frac{1}{k}\sum_{j=1}^{N}
            t_j\boldsymbol y^{j}+\hat{\boldsymbol\epsilon}_i\right)
            \ \ge \
            \left(\boldsymbol\pi_i\right)^{\top}
            \!\left[H (\boldsymbol z)+\boldsymbol T\!\left(
            \frac{1}{k}\sum_{j=1}^{N}t_j\boldsymbol y^{j}
            +\hat{\boldsymbol\epsilon}_i\right)\right] \). Hence, \eqref{eq:lower_bound_theta}--\eqref{eq:Bender's_cuts} provide a relaxation of \eqref{eq:full_second_approximation}. Together
        with the fact that the RMP~\eqref{eq:master_problem} enforces only a subset of the \(k\)NN
    selection constraints, the RMP~\eqref{eq:master_problem} is a relaxation of the full problem~\eqref{eq:full_problem}.
        Consequently, for this minimization problem, its optimal value
        provides a valid lower bound (LB) on the optimal value of
        problem~\eqref{eq:full_problem}.
 
        \item Upon termination at iteration \(m\), the exact termination
        criterion gives \(F^m \subseteq S^m\). Hence the RMP-selected \(k\)NN coincides with the true \(k\)NN set, which means that the
        current solution \((\boldsymbol z^m,\boldsymbol t^m)\) is feasible to
        the problem~\eqref{eq:full_problem}. Suppose
        \((\boldsymbol z^{\ast},\boldsymbol t^{\ast},\theta^{\ast})\) is an
        optimal solution of the problem~\eqref{eq:full_problem}.
        With weak and strong duality of LP, we have
        \begin{align*}
            \theta^{\ast}
             &\ge
            \frac{1}{N}\sum_{i=1}^{N}
            h\!\left(\boldsymbol z^{\ast},\ \frac{1}{k}\sum_{j=1}^{N}
            t^{\ast}_j\boldsymbol y^{j}+\hat{\boldsymbol\epsilon}_i\right)
            \ =\
            \frac{1}{N}\sum_{i=1}^{N} \max_{\boldsymbol\pi_i}\ 
            \left(\boldsymbol\pi_i\right)^{\top}
            \!\left[H (\boldsymbol z^{\ast})+\boldsymbol T\!\left(
            \frac{1}{k}\sum_{j=1}^{N}t^{\ast}_j\boldsymbol y^{j}
            +\hat{\boldsymbol\epsilon}_i\right)\right]
            \\
            &\ge
            \Theta^{\ell}(\boldsymbol z^{\ast},\boldsymbol t^{\ast}), \, \forall \ell \in [m].
        \end{align*}
        Therefore,
        \((\boldsymbol z^{\ast},\boldsymbol t^{\ast},\theta^{\ast})\) is
        feasible for the RMP at iteration \(m\), which implies
        \(
            \boldsymbol f^{\top}\boldsymbol z^{m}+\theta^{m}
            \ \le\
            \boldsymbol f^{\top}\boldsymbol z^{\ast}+\theta^{\ast}.
        \)
        Moreover, when the algorithm terminates, the current solution satisfies the newly generated cut,
        and by strong duality of the subproblem~\eqref{eq:sub_problem},
        \[
            \theta^{m}
            \ \ge\
            \Theta^{m}(\boldsymbol z^{m},\boldsymbol t^{m})
            \ =\
            \frac{1}{N}\sum_{i=1}^{N}
            h\!\left(\boldsymbol z^{m},\ \frac{1}{k}\sum_{j=1}^{N}
            t^{m}_j\boldsymbol y^{j}+\hat{\boldsymbol\epsilon}_i\right).
        \]
        Since \((\boldsymbol z^{m},\boldsymbol t^{m})\) is feasible for the problem~\eqref{eq:full_problem}, its objective value provides a valid upper bound (UB),
        i.e.,
        \(\boldsymbol f^{\top}\boldsymbol z^{m}
        +\frac{1}{N}\sum_{i=1}^{N}h(\boldsymbol z^{m},
        \frac{1}{k}\sum_{j=1}^{N}t^{m}_j\boldsymbol y^{j}
        +\hat{\boldsymbol\epsilon}_i)
        \ge \boldsymbol f^{\top}\boldsymbol z^{\ast}+\theta^{\ast}\).
        Combining the above inequalities yields
        \[
            \boldsymbol f^{\top}\boldsymbol z^{\ast}+\theta^{\ast}
            \ \le\
            \boldsymbol f^{\top}\boldsymbol z^{m}
            +\frac{1}{N}\sum_{i=1}^{N}
            h\!\left(\boldsymbol z^{m},\ \frac{1}{k}\sum_{j=1}^{N}
            t^{m}_j\boldsymbol y^{j}+\hat{\boldsymbol\epsilon}_i\right)
            \ \le\
            \boldsymbol f^{\top}\boldsymbol z^{m}+\theta^{m}
            \ \le\
            \boldsymbol f^{\top}\boldsymbol z^{\ast}+\theta^{\ast},
        \]
        which shows that $(\boldsymbol{z}^m, \boldsymbol{t}^m)$ is optimal to \eqref{eq:full_problem} at termination.
    \end{enumerate}
 
    \item \textbf{Finite convergence.}
    \begin{enumerate}
        \item The full formulation of the \(k\)NN selection contains only
        finitely many constraints. In particular, the total number of such
        constraints added in the algorithm under the pairwise distance
        comparison is at most \(\frac{N^2+N}{2}\) for
        \eqref{eq:random_select_knn_strict1}--\eqref{eq:random_select_knn_strict3}.
 
        \item Since the second-stage LP has relatively complete and
        sufficiently expensive recourse, the number of extreme points of its
        dual polyhedron is finite, meaning that only finitely many Bender's
        optimality cuts can be generated.
 
        \item Consider the incumbent solution
        \((\boldsymbol z^m,\boldsymbol t^m,\theta^m)\) obtained by solving
        the RMP~\eqref{eq:master_problem} at iteration \(m\). If the
        algorithm has not terminated, the current solution must be strictly
        cut off by at least one of the following two mechanisms:
        \begin{itemize}
            \item \textbf{Constraint Generation.}
            For the pairwise distance comparison formulation, an invalid
            \(k\)NN selection implies that there exist indices
            \(i \in S^m \setminus F^m\) and \(j \in F^m \setminus S^m\) for
            which the incumbent solution violates the corresponding ordering
            constraint
            \(
                M_1^{ij}(d_{ij}-1) \le s_i - s_j \le M_1^{ij} d_{ij}.
            \)
            Adding this violated constraint to the
            RMP~\eqref{eq:master_problem} therefore strictly cuts off the
            current infeasible incumbent solution. Since the active set can be enlarged
            only up to the finite set of constraints in the full formulation,
            this cutoff mechanism can occur only finitely many times.
 
            \item \textbf{Bender's Decomposition.}
            Suppose that the current incumbent solution
            \((\boldsymbol z^m,\boldsymbol t^m,\theta^m)\) does not satisfy
            the newly generated cut, i.e.,
            \(
                \theta^m < \Theta^m(\boldsymbol z^m,\boldsymbol t^m).
            \)
            Then, adding the following Bender's optimality
            cut~\eqref{eq:dual_cut} will cut off this solution:
            \(
                \theta^m \ \ge\ \Theta^m(\boldsymbol z^m,\boldsymbol t^m).
            \)
            Since the number of dual extreme points is finite under
            relatively complete and sufficiently expensive recourse, only
            finitely many distinct Bender's optimality cuts can be generated.
        \end{itemize}
        Because both the pool of candidate \(k\)NN constraints and the number
        of extreme points of the second-stage dual polyhedron are finite, the
        algorithm must terminate in finitely many iterations.
    \end{enumerate}
\end{enumerate}
This completes the proof. 
\end{proof}

Next, we verify that the three nonparametric regression models considered in this paper (i.e., $k$NN, CART, ReLU NNs) satisfy the uniform consistency in Assumption \ref{assumption:estimator}.
\begin{theorem}[Consistency of \(k\)NN] \label{lem:knn_convergence} (adapted from \cite[Theorem 12.1]{biau2015lectures})
Under the following conditions: (i)The space $\mathcal{Z} \times \mathcal{W}$ is compact; (ii) $Q^\ast(\boldsymbol{z}, \boldsymbol{w})$ is continuous on $\mathcal{Z} \times \mathcal{W}$; (iii) There exists $\lambda > 0$ such that
    \(
    \sup_{(\boldsymbol{z}, \boldsymbol{w}) \in \mathcal{Z} \times \mathcal{W}} 
    \mathbb{E}\!\left[ e^{\lambda |Y-Q^\ast(\boldsymbol{z}, \boldsymbol{w})|} \,\middle|\, \boldsymbol{Z} = \boldsymbol{z}, \boldsymbol{W} = \boldsymbol{w} \right] < \infty;
    \) (iv) The number of neighbors \(k_N\) satisfies $k_N \to \infty$, $k_N/N \to 0$, and $k_N / \log N \to \infty$.
Then, we have $\sup_{(\boldsymbol{z},\boldsymbol{w})\in \mathcal{Z}\times\mathcal{W}}
\left| \hat{Q}_N^{k\text{-}\mathrm{NN}}(\boldsymbol{z},\boldsymbol{w})
- Q^*(\boldsymbol{z},\boldsymbol{w}) \right|
\xrightarrow{\text{a.s.}} 0$.

\end{theorem}

\begin{theorem}[Consistency of CART]\label{lem:cart-uniform} (adapted from {\cite[Lemma 7]{bertsimas2019predictions}})
Let \((\boldsymbol z,\boldsymbol w)\) take values in
\(\mathcal Z\times\mathcal W = [0,1]^{d_z + d_w}\), suppose that: (i) \(\hat Q_N(\boldsymbol z,\boldsymbol w)\) is a regular,
    random-split, honest tree trained on i.i.d. samples; (ii) \(Q^*(\boldsymbol z,\boldsymbol w)\) is Lipschitz continuous on
    \(\mathcal Z\times\mathcal W\); (iii) there exists \(\lambda>0\) such that
    \(
    \sup_{(\boldsymbol z,\boldsymbol w)\in\mathcal Z\times\mathcal W}
    \mathbb E\!\left[
        \exp\!\left(
            \lambda \left|Y-Q^*(\boldsymbol z,\boldsymbol w)\right|
        \right)
        \,\middle|\, \boldsymbol{Z}=\boldsymbol z, \boldsymbol{W} = \boldsymbol w)
    \right]
    <\infty;
    \) (iv) the tree is grown to full depth \(n_d \in\mathbb N\),
    where \(\log (N) /n_d\to 0\), \(N/n_d \to \infty,\) as \(N\to \infty.\) 
Then, we have \(
\sup_{(\boldsymbol z,\boldsymbol w)\in\mathcal Z\times\mathcal W}
\left|
\hat Q^{\text{CART}}_N(\boldsymbol z,\boldsymbol w)-Q^*(\boldsymbol z,\boldsymbol w)
\right|
\xrightarrow{P}0.
\)
\end{theorem}

\begin{theorem}(Adapted from {\cite[Theorem 2]{imaizumi2023sup}})\label{thm:dnn-Linfty-mean}
Let \((\boldsymbol z,\boldsymbol w)\) take values in
\(\mathcal Z\times\mathcal W = [0,1]^{d_z + d_w}\) and \(\boldsymbol{Y} \in \mathbb{R}^{d_y}\). Let \(\hat Q_N^{\mathrm{NNs}}\) denote the adversarial estimator over deep NNs with
depth \(N_d\) and width \(N_w\). Suppose that: (i) the marginal
measure \(P_{\mathcal{Z} \times \mathcal{W}}\) has a density uniformly lower bounded by \(C_{P_{\mathcal{Z} \times \mathcal{W}}}>0\). (ii) \(Q^*\) is continuous, \(\mathbb{E}[\|\hat Q_N^{\mathrm{NNs}}\|_{L^\infty}^2] \le V^2\) for some \(V>0\), and \(\mathbb{E}[\|\hat Q_N^{\mathrm{NNs}} - Q^*\|_{L^\infty}^2] \le \zeta_N^2\), where \(\zeta_N \ge 0\) and \(\zeta_N \to 0\) as \(N \to \infty\).  Then there exists \((N_d,N_w)\) with
\(N_dN_w=o(N)\) such that
\(
\mathbb E\!\left[\|\hat Q_N^{\mathrm{NNs}} -Q^*\|_{L^\infty}^2\right]\to 0,
\, \text{as } N \to\infty.
\)
\end{theorem}

\begin{theorem}[Consistency of ReLU NNs]\label{thm:nn_uniform}
Under the assumptions of Theorem~\ref{thm:dnn-Linfty-mean}, we have
\[
\sup_{(\boldsymbol z,\boldsymbol w)\in\mathcal Z\times\mathcal W}
\left|\hat Q_N^{\mathrm{NNs}} (\boldsymbol z,\boldsymbol w)-Q^*(\boldsymbol z,\boldsymbol w)\right|
\xrightarrow{P} 0.
\]
\end{theorem}

\begin{proof}{Proof}
By Theorem~\ref{thm:dnn-Linfty-mean}, there exists a
choice of network depth \(N_d\) and width \(N_w\) such that
\(
\mathbb E\!\left[\|\hat Q_N^{\mathrm{NNs}} -Q^*\|_{L^\infty}^2\right]\to 0
\qquad \text{as } N \to\infty.
\)
By Markov's inequality, for any \(\epsilon>0\),
\(
\mathbb P\!\left(
\|\hat Q_N^{\mathrm{NNs}}-Q^*\|_{L^\infty}>\epsilon
\right)
\le
\frac{\mathbb E\!\left[\|\hat Q_N^{\mathrm{NNs}}-Q^*\|_{L^\infty}^2\right]}{\epsilon^2}.
\)
Hence,
\(
\mathbb P\!\left(
\|\hat Q_N^{\mathrm{NNs}}-Q^*\|_{L^\infty}>\epsilon
\right)\to 0,
\qquad \forall \epsilon>0,
\)
which implies
\(
\|\hat Q_N^{\mathrm{NNs}}-Q^*\|_{L^\infty}\xrightarrow{P}0.
\)
Equivalently,
\(
\sup_{(\boldsymbol z,\boldsymbol w)\in\mathcal Z\times\mathcal W}
\left|
\hat Q_N^{\mathrm{NNs}} (\boldsymbol z,\boldsymbol w)-Q^*(\boldsymbol z,\boldsymbol w)
\right|
\xrightarrow{P}0.
\) This completes the proof. 
\end{proof}


\section{Parameters tuning}\label{sec:parameters_tuning}
Table~\ref{tab:parameter_tuning} summarizes the hyperparameter tuning procedures before solving the ER-DD-SAA problem. For all models, the
hyperparameters are selected by five-fold cross-validation using negative mean
squared error as the scoring criterion.  For the two-stage facility location
problem, the candidate hidden-layer configurations \((16,32)\) and \((32,32)\)
are excluded to reduce the size of the resulting MIP formulation.

\begin{table}[H]
\centering
\caption{Summary of hyperparameter tuning for regression models}
\label{tab:parameter_tuning}
\begin{tabular}{lll}
\hline
Model & Preprocessing & Tuned hyperparameters \\ \hline
$k$NN & \texttt{StandardScaler} & $k \in \{1,2,3\}$ \\ \hline
CART & None &
\begin{tabular}[t]{@{}l@{}}
\texttt{max\_depth} $\in\{3, 6, 9, 12\}$ \\
\texttt{min\_samples\_split} $\in\{5, 10, 15\}$ \\
\texttt{min\_samples\_leaf} $\in\{2, 4, 6, 8\}$
\end{tabular} \\ \hline
ReLU NNs & \texttt{StandardScaler} &
\begin{tabular}[t]{@{}l@{}}
\texttt{hidden\_layer\_sizes} $\in \{(8),(16),(32),(8,8),(8,16),$ \\
\qquad $(8,32),(16,16),(16,32),(32,32)\}$ \\
\texttt{activation} $=$ \texttt{relu} \\
\texttt{alpha} $\in\{10^{-4},10^{-3},10^{-2}\}$ \\
\texttt{solver} $\in\{\texttt{lbfgs},\texttt{adam}\}$ \\
For \texttt{lbfgs}: \texttt{max\_iter} $=50000$ \\
For \texttt{adam}: \texttt{learning\_rate\_init} $\in\{10^{-3},10^{-2}\}$ \\
For \texttt{adam}: \texttt{early\_stopping} $=$ \texttt{True} \\
For \texttt{adam}: \texttt{validation\_fraction} $=0.1$ \\
For \texttt{adam}: \texttt{max\_iter} $=5000$
\end{tabular} \\ \hline
\end{tabular}
\end{table}

\section{Full ER-DD-SAA Formulations} \label{sec:full_formulation}
\subsection{Full ER-DD-SAA formulations for newsvendor problem with pricing}\label{subsec:newsvendor_full_formulation}
Let \(\overline{P}=20\) and \(\underline{P}=0\) denote valid upper and lower bounds for pricing decision \(p\). Let \(\overline{v}\) denote the maximum potential revenue, estimated as \(\overline{P} \cdot \max_{i \in [N]} y^i\). The variables $s_i^+$ and $s_i^-$ capture the positive and negative parts of 
$p - p^i$, respectively, whereas the binary variable $\Omega_i$ indicates 
whether $p - p^i$ is nonnegative.  The full ER-DD-SAA with \(k\)NN model formulation \eqref{eq:distance_compare_tie} for the newsvendor problem with pricing is given by the following MILP:
\begin{subequations}\label{eq:newsvendor_knn}
    \begin{align}
\min \quad & \frac{1}{N} \sum_{i=1}^N (f \cdot q-v_i + h \cdot h_i + b \cdot b_i )\\
\text{s.t.} \quad & \eqref{eq:random_select_knn_sum_t}--\eqref{eq:comparison_binary} \\
& s^+_{i} - s^-_{i} = p - p^i, \forall i \in [N], \\
& s_i =s^+_{i} + s^-_{i}+ |\boldsymbol{w}^i-\boldsymbol{w}|,\quad \forall i \in [N],\\
& 0\le s^+_{i} \leq M \Omega_{i}, \quad \forall i \in [N],\\
& 0\le s^-_{i} \leq M (1-\Omega_{i}), \quad \forall i \in [N],\\
&p - \overline{P} (1-t_i) \leq \phi_i \leq p - \underline{P} (1-t_i), \quad \forall i \in [N], \label{eq:phi_mccormick_1}\\
& \underline{P} t_i \le \phi_i \leq \overline{P} t_i, \quad \forall i \in [N], \\
&  \frac{1}{k}\sum_{i=1}^N \phi_i y^i + p \cdot \hat{\epsilon}^i \le v_i \le \frac{1}{k}\sum_{i=1}^N \phi_i y^i + p \cdot \hat{\epsilon}^i + \bar{v} (1-g_i) , \quad \forall i \in [N], \label{eq:linearize_revenu1}\\
& 0 \le v_i \le \bar{v} g_i,  \forall i \in [N], \\
& \hat{y} = \frac{1}{k}\sum_{j=1}^N t_j y^j, \label{eq:point_prediction}\\
&\hat{y} + \hat{\epsilon}^i\le D_i \le \hat{y} + \hat{\epsilon}^i + M(1-g_i), \quad \forall i \in [N], \label{eq:projection_prediction} \\
& 0\le D_i \leq M g_i, \quad \forall i \in [N],\\
& h_i \geq q - D_i, \quad \forall i \in [N],\\
& b_i \geq D_i - q, \quad \forall i \in [N],\\
& \underline{P} \leq p \leq \bar{P}, \\
& q \ge 0, \quad q \in \mathbb{Z}, \\
& g_i \in \{0, 1\}, \quad \forall i \in [N], \\
& D_i, h_i, b_i, \phi_i \geq 0, \quad \forall i \in [N], \label{eq:positive_variables}\\
& \Omega_i \in \{0, 1\}, \quad \forall i \in [N].
\end{align}
\end{subequations}
The ER-DD-SAA with $k$NN formulation \eqref{eq:single_level_binary} can be obtained by replacing \eqref{eq:random_select_knn_sum_t}--\eqref{eq:comparison_binary} in Model \eqref{eq:newsvendor_knn} by \eqref{eq:bilevel-1}--\eqref{eq:bilevel-last}.

The ER-DD-SAA with CART \eqref{eq:er_dd_saa_cart} for the newsvendor problem with pricing is given by the following MILP:
\begin{subequations}\label{eq:newsvendor_cart}
    \begin{align}
\min \quad & \frac{1}{N} \sum_{i=1}^N (f \cdot q-v_i + h \cdot h_i + b \cdot b_i )\\
\text{s.t.} \quad & \eqref{eq:ind_region_lowerbound1}--\eqref{eq:cart_binary_decision} \\
& p - \overline{P} (1-R_r) \le  \phi_r \leq p - \underline{P} (1-R_r), \quad \forall r \in \mathcal{N}^{\boldsymbol{w}}, \\
& \underline{P} R_r \le \phi_r \leq \overline{P} R_r, \quad \forall r \in \mathcal{N}^{\boldsymbol{w}}, \\
&  \sum_{r \in \mathcal{N}^{\boldsymbol{w}}} \phi_r \hat{y}^r + p \cdot \hat{\epsilon}^i \le v_i \le \sum_{r \in \mathcal{N}^{\boldsymbol{w}}} \phi_r \hat{y}^r + p \cdot \hat{\epsilon}^i + \bar{v} (1-g_i) , \quad \forall i \in [N],\\
& 0 \le v_i \le \bar{v} g_i,\quad  \forall i \in [N], \\
& \hat{y} = \sum_{r \in \mathcal{N}^{\boldsymbol{w}}} R_r \hat{y}^r, \\
& \eqref{eq:projection_prediction} --  \eqref{eq:positive_variables}.
\end{align}
\end{subequations}

The ER-DD-SAA with ReLU NNs \eqref{eq:nn_erddsaa} for the newsvendor problem with pricing is given by:
\begin{subequations}\label{eq:newsvendor_nn}
    \begin{align}
\min \quad & \frac{1}{N} \sum_{i=1}^N (f \cdot q-v_i + h \cdot h_i + b \cdot b_i )\\
\text{s.t.} \quad & \eqref{eq:nn_first_layer}, \eqref{eq:nn_output_layer}, \eqref{eq:relu_linear_1} -- \eqref{eq:relu_linear_4}\\
& v_i = \max \{ p \cdot (\hat{y} + \hat{\epsilon}^i ), 0\}, \quad \forall i \in [N], \\
& \eqref{eq:projection_prediction} -- \eqref{eq:positive_variables}.
\end{align}
\end{subequations}
which involves the product $p \cdot \hat{y}$ with two continuous decision variables. Therefore, Model \eqref{eq:newsvendor_nn} is an MINLP.

\subsection{Two-Stage Facility Location Problem}\label{subsec:facility_formulation}
In the two-stage facility location problem, the DM decides which facilities to open in the first stage. In the second-stage problem, let $x_{ij}$ denote the shipment quantity from facility $i \in \Gamma_1$ to customer site $j \in \Gamma_2$, and $s_j$ the unmet demand at site $j$. 
For a demand realization $\boldsymbol{Y}(\boldsymbol{z},\boldsymbol{w})$, the second-stage recourse function is:
\begin{subequations}\label{eq:sec_stage}
\begin{align}
    h\big(\boldsymbol{z},\boldsymbol{Y}(\boldsymbol{z},\boldsymbol{w})\big)
    =
    \min_{\boldsymbol{x},\boldsymbol{s}} \quad
    & \sum_{i\in \Gamma_1}\sum_{j\in \Gamma_2} c_{ij}x_{ij}
    + \sum_{j\in \Gamma_2}
    \left(p_j s_j - r_j \boldsymbol{Y}(\boldsymbol{z},\boldsymbol{w}) \right)  
    \label{eq:sec_stage_obj} \\
    \text{s.t.} \quad
    & \sum_{i\in \Gamma_1} x_{ij} + s_j
    \ge \boldsymbol{Y}(\boldsymbol{z},\boldsymbol{w}),
    \quad \forall j\in \Gamma_2, \label{eq:demand_cons} \\
    & \sum_{j\in \Gamma_2} x_{ij}
    \le C_i z_i,
    \quad \forall i\in \Gamma_1, \label{eq:capacity_cons} \\
    & x_{ij}, s_j \ge 0, \quad \forall i\in \Gamma_1, j\in \Gamma_2.
\end{align}
\end{subequations}

In the downstream optimization problem, fixed opening costs and facility capacities are sampled independently as $f_i \sim U (20000,30000)$ and $C_i \sim U(40000,60000)$, respectively, for all $i\in \Gamma_1$. Revenues and penalty costs are sampled as $r_j \sim U(5,10)$ and $p_j \sim U(1,5)$, respectively, for all $j\in \Gamma_2$. Facility and customer locations are generated independently and uniformly over the $20 \times 20$ grid. The unit transportation cost is set to $0.05$ times the Euclidean distance. For simplicity, we assume a common gas price across all customer sites, set to $e=4$. This avoids introducing multiple predictors into the downstream optimization problem. $\chi_1$ and $\chi_2$ are randomly drawn parameters with
$\chi_1 \sim \mathrm{Uniform}(15000,\, 16000)$ serving as the initial demand and
$\chi_2 \sim \mathrm{Uniform}(475,\, 525)$ scaling the joint effect of facility openings and gas price,
and $\chi_3 = 30$ is the maximum number of open facilities.
The resulting ER-DD-SAA problem for \eqref{eq:two_stage_facility_location_problem} is
\begin{subequations}\label{eq:er_dd_saa_facility_full}
\begin{align}
    \min_{\boldsymbol{z}, \boldsymbol{x},\boldsymbol{s}} \quad
    & \sum_{i \in \Gamma_1} f_i z_i + \frac{1}{N}\sum_{n=1}^N \left( \sum_{i\in \Gamma_1}\sum_{j\in \Gamma_2} c_{ij}x_{ijn}
    + \sum_{j\in \Gamma_2}
    \left(p_j s_{jn} - r_j (\hat{Q}_N(\boldsymbol{z},\boldsymbol{w})+\hat{\epsilon}_n) \right)  \right)
    \label{eq:sec_stage_obj_full} \\
    \text{s.t.} \quad
    & \sum_{i\in \Gamma_1} x_{ijn} + s_{jn}
    \ge \hat{Q}_N(\boldsymbol{z},\boldsymbol{w}) + \hat{\epsilon}_n,
    \quad \forall j\in \Gamma_2,\ \forall n \in [N],  \label{eq:demand_cons_full} \\
    & \sum_{j\in \Gamma_2} x_{ijn}
    \le C_i z_i,
    \quad \forall i\in \Gamma_1,\ \forall n \in [N], \label{eq:capacity_cons_full} \\
    & x_{ijn}, s_{jn} \ge 0, \quad \forall i\in \Gamma_1,\ j\in \Gamma_2,\ \forall n \in [N], \\
    & z_i \in \{0,1\},\ \forall i \in \Gamma_1. \label{eq:facility_binary_full}
\end{align}
\end{subequations}

The full formulation for ER-DD-SAA with \(k\)NN for the two-stage facility locaion problem \eqref{eq:two_stage_facility_location_problem} is formulated as:
\begin{subequations}\label{eq:er_dd_saa_facility_full_knn}
\begin{align}
     \min_{\boldsymbol{z}, \boldsymbol{x},\boldsymbol{s}} \quad
    & \sum_{i \in \Gamma_1} f_i z_i + \frac{1}{N}\sum_{n=1}^N \left( \sum_{i\in \Gamma_1}\sum_{j\in \Gamma_2} c_{ij}x_{ijn}
    + \sum_{j\in \Gamma_2}
    \left(p_j s_{jn} - r_j (\hat{Q}_N^{k\text{-}\mathrm{NN}}(\boldsymbol{z},\boldsymbol{w})+\hat{\epsilon}_n) \right)  \right) \\
    \text{s.t.} \quad
    & \eqref{eq:demand_cons_full}--\eqref{eq:facility_binary_full},\  \eqref{eq:random_select_knn_p_norm}--\eqref{eq:comparison_binary}, \\
    & \hat{Q}_N^{k\text{-}\mathrm{NN}}(\boldsymbol{z}, \boldsymbol{w}) = \frac{1}{k} \sum_{l=1}^N t_l y^l.
\end{align}
\end{subequations}

The full formulation for ER-DD-SAA with CART for the two-stage facility locaion problem \eqref{eq:two_stage_facility_location_problem} is formulated as:
\begin{subequations}\label{eq:er_dd_saa_facility_full_cart}
\begin{align}
     \min_{\boldsymbol{z}, \boldsymbol{x},\boldsymbol{s}} \quad
    & \sum_{i \in \Gamma_1} f_i z_i + \frac{1}{N}\sum_{n=1}^N \left( \sum_{i\in \Gamma_1}\sum_{j\in \Gamma_2} c_{ij}x_{ijn}
    + \sum_{j\in \Gamma_2}
    \left(p_j s_{jn} - r_j (\hat{Q}_N^{\text{CART}}(\boldsymbol{z},\boldsymbol{w})+\hat{\epsilon}_n) \right)  \right) \\
    \text{s.t.} \quad
    & \eqref{eq:demand_cons_full}--\eqref{eq:facility_binary_full}, \\ &\eqref{eq:ind_region_lowerbound1}--\eqref{eq:cart_binary_decision}, \\
    & \hat{Q}_N^{\text{CART}}(\boldsymbol{z}, \boldsymbol{w}) = \sum_{r \in \mathcal{N}^{\boldsymbol{w}}} R_r \hat{y}^r.
\end{align}
\end{subequations}

The full formulation for ER-DD-SAA with ReLU NNs for the two-stage facility location problem \eqref{eq:two_stage_facility_location_problem} is given by:
\begin{subequations}\label{eq:er_dd_saa_facility_full_nn}
\begin{align}
     \min_{\boldsymbol{z}, \boldsymbol{x},\boldsymbol{s}} \quad
    & \sum_{i \in \Gamma_1} f_i z_i + \frac{1}{N}\sum_{n=1}^N \left( \sum_{i\in \Gamma_1}\sum_{j\in \Gamma_2} c_{ij}x_{ijn}
    + \sum_{j\in \Gamma_2}
    \left(p_j s_{jn} - r_j (\hat{Q}_N^{\text{NNs}} (\boldsymbol{z},\boldsymbol{w})+\hat{\epsilon}_n) \right)  \right) \\
    \text{s.t.} \quad
    & \eqref{eq:demand_cons_full}--\eqref{eq:facility_binary_full},\\  &\eqref{eq:nn_first_layer}, \eqref{eq:nn_output_layer},  \eqref{eq:relu_linear_1} -- \eqref{eq:relu_linear_4}, \\
    & \hat{Q}_N^{\text{NNs}}(\boldsymbol{z}, \boldsymbol{w}) = \hat{y}.
\end{align}
\end{subequations}
All three formulations above are MILPs.

\section{Additional Numerical Results}\label{sec:additional_results}

\subsection{Comparison between different models for ER-DD-SAA with \(k\)NN in newsvendor problem with pricing}\label{sec:knn_comparison}
We compare three formulations for ER-DD-SAA with \(k\)NN: the pairwise distance formulation~\eqref{eq:distance_compare_tie}, the bilevel reformulation~\eqref{eq:single_level_binary}, and the benchmark \eqref{eq:zhihai}. We set \texttt{NumericFocus}=3 and \texttt{PoolSearchMode}=2 to improve numerical stability. Table~\ref{tb:knn_algorithm} shows that the bilevel reformulation~\eqref{eq:single_level_binary} is generally the fastest and achieves competitive OOS performance. The pairwise distance formulation~\eqref{eq:distance_compare_tie} becomes computationally expensive as \(N\) increases and fails to reach optimality for some larger instances, leading to worse performance. The benchmark formulation~\eqref{eq:zhihai} also becomes substantially slower as \(N\) and \(k\) increase; for \(N=500\) and \(k=2\), only three of five instances are solved to optimality. Overall, the bilevel reformulation~\eqref{eq:single_level_binary} provides the best computational performance among the formulations considered. 

\begin{table}[htbp]
\centering
\caption{Comparison of solution approaches for ER-DD-SAA with $k$NN in newsvendor problem with pricing}
\label{tb:knn_algorithm}
\resizebox{\textwidth}{!}{%
\begin{tabular}{c|c|rrr|rrr|rrr}
\hline
\multirow{2}{*}{$N$}
& \multirow{2}{*}{$k$}
& \multicolumn{3}{c|}{Model~\eqref{eq:distance_compare_tie}}
& \multicolumn{3}{c|}{Bilevel~\eqref{eq:single_level_binary}}
& \multicolumn{3}{c}{Benchmark~\eqref{eq:zhihai}} \\
\cline{3-5} \cline{6-8} \cline{9-11}
& & \multicolumn{1}{c}{Time (s)} & \multicolumn{1}{c}{IS Cost} & \multicolumn{1}{c|}{OOS Cost}
& \multicolumn{1}{c}{Time (s)} & \multicolumn{1}{c}{IS Cost} & \multicolumn{1}{c|}{OOS Cost}
& \multicolumn{1}{c}{Time (s)} & \multicolumn{1}{c}{IS Cost} & \multicolumn{1}{c}{OOS Cost} \\ \hline
100 & \multirow{3}{*}{1}
& 11.36     & -58,296.68 & -40,969.97
& \textbf{1.85} & -58,296.68 & -40,969.97
& 5.12      & -58,296.68 & -40,969.97 \\
300 &
& 1,195.00  & -56,870.09 & -50,834.58
& \textbf{35.42} & -56,870.09 & -50,834.58
& 173.13    & -56,870.09 & -50,834.58 \\
500 &
& 10,429.55 & -41,363.26 & -39,421.56
& \textbf{268.00} & -56,419.17 & -51,993.29
& 1,190.71  & -56,419.17 & -51,992.29 \\ \hline
100 & \multirow{3}{*}{2}
& 28.98     & -57,190.38 & -45,501.79
& \textbf{6.54}  & -57,190.38 & -45,501.79
& 8.92      & -57,190.38 & -45,501.79 \\
300 &
& 3,666.56  & -57,056.02 & -50,407.23
& 318.23    & -57,056.02 & -50,407.23
& \textbf{294.46} & -57,056.02 & -50,407.23 \\
500 &
& 10,821.29 & -5,426.97  & -3,914.72
& \textbf{684.96} & -56,133.16 & -52,633.38
& 6,991.87  & -53,869.78 & -51,107.66 \\ \hline
\end{tabular}%
}
\end{table}

\subsection{Comparison between different models for ER-DD-SAA with CART in newsvendor problem with pricing}\label{sec:cart_comparison}
Table~\ref{tab:cart_comparison} compares ER-DD-SAA with CART under three optimization formulations: the proposed general formulation~\eqref{eq:er_dd_saa_cart}, the single-dimensional formulation~\eqref{eq:special_cart}, and the direct implementation by the Gurobi ML package \citep{gurobiMLfeatures}. Since \(p\) is continuous and formulations in~\eqref{eq:er_dd_saa_cart} and ~\eqref{eq:special_cart} rely on Big-M constraints, we set the Gurobi parameter \texttt{NumericFocus}=$3$  to enhance numerical stability and use default settings for all other parameters. Table~\ref{tab:cart_comparison} shows that all three formulations obtain nearly identical IS and OOS costs. For larger sample sizes, the proposed MILP formulations~\eqref{eq:er_dd_saa_cart} and \eqref{eq:special_cart} are much faster than Gurobi ML. Among the two proposed formulations, Model~\eqref{eq:er_dd_saa_cart} is generally faster, suggesting that the additional constraints in Model~\eqref{eq:special_cart} can offset the benefit of using fewer variables.

\begin{table}[htbp]
\centering
\caption{Computational comparisons for ER-DD-SAA with CART in newsvendor problem with pricing.}
\label{tab:cart_comparison}
\resizebox{\textwidth}{!}{%
\begin{tabular}{c|c|rrr|rrr|rrr}
\hline
\multirow{2}{*}{$N$} & \multirow{2}{*}{$|\mathcal{N}^{w}|$}
& \multicolumn{3}{c|}{Model~\eqref{eq:er_dd_saa_cart}}
& \multicolumn{3}{c|}{Model~\eqref{eq:special_cart}}
& \multicolumn{3}{c}{Gurobi ML} \\
\cline{3-5} \cline{6-8} \cline{9-11}
& & \multicolumn{1}{c}{IS Cost} & \multicolumn{1}{c}{OOS Cost} & \multicolumn{1}{c|}{Time (s)}
& \multicolumn{1}{c}{IS Cost} & \multicolumn{1}{c}{OOS Cost} & \multicolumn{1}{c|}{Time (s)}
& \multicolumn{1}{c}{IS Cost} & \multicolumn{1}{c}{OOS Cost} & \multicolumn{1}{c}{Time (s)} \\ \hline
1000 & 43.60 & -58,043.45 & -51,500.96 & 19.09           & -58,043.53 & -51,500.91 & 18.42  & -58,043.45 & -51,500.96 & \textbf{3.46} \\
3000 & 49.00 & -57,546.12 & -52,500.47 & \textbf{85.78}  & -57,546.12 & -52,501.47 & 107.71 & -57,546.12 & -52,501.47 & 315.25 \\
5000 & 41.80 & -57,462.58 & -52,148.86 & \textbf{204.40} & -57,462.58 & -52,148.86 & 222.27 & -57,462.58 & -52,148.86 & 911.60 \\
7000 & 30.60 & -56,680.22 & -53,481.89 & \textbf{308.63} & -56,680.24 & -53,481.88 & 419.53 & -56,680.22 & -53,481.89 & 1,590.57 \\ \hline
\end{tabular}%
}
\end{table}

\subsection{Performance of ER-DD-SAA with ReLU NNs in newsvendor problem with pricing}\label{sec:nn_comparison_newsvendor}

For each sample size \(N\), a ReLU network is trained using \texttt{MLPRegressor} in \texttt{scikit-learn} and embedded into ER-DD-SAA using Gurobi ML. Table~\ref{tab:nn_sample_size_comparison} shows that the OOS cost improves from \(N=1000\) to \(N=3000\) and remains relatively stable thereafter. Both training and optimization times increase with \(N\), with optimization time rising sharply beyond \(N=3000\). Thus, larger training data set provide limited additional OOS improvement while substantially increasing computational cost.

\begin{table}[htbp]
\centering
\caption{Computational comparison of ER-DD-SAA with ReLU NNs in newsvendor problem with pricing.}
\label{tab:nn_sample_size_comparison}
\setlength{\tabcolsep}{24pt}
\begin{tabular}{lrrrr}
\hline
$N$
& \multicolumn{1}{c}{IS Cost}
& \multicolumn{1}{c}{OOS Cost}
& \multicolumn{1}{c}{Train}
& \multicolumn{1}{c}{Opt.}
\\
\hline
1000 & -55,413.35 & -55,136.74 & 47.91  & 45.82 \\
3000 & -55,458.69 & -55,230.75 & 112.95 & 185.21 \\
5000 & -55,443.53 & -55,214.43 & 182.77 & 691.68 \\
7000 & -55,448.63 & -55,120.99 & 259.18 & 1,380.22 \\
\hline
\end{tabular}
\end{table}

\subsection{Comparisons between different models for ER-DD-SAA with CART in two-stage facility location}\label{sec:facility_cart_comparison}
We compare the general formulation~\eqref{eq:er_dd_saa_cart}, the one-dimensional formulation~\eqref{eq:special_cart}, and the direct implementation using the Gurobi ML package~\citep{gurobiMLfeatures} in the two-stage facility location problem.
As shown in Table~\ref{tab:cart_comparison_facility}, all three formulations yield the same IS and OOS costs. Formulation~\eqref{eq:er_dd_saa_cart} is consistently the fastest. Although formulation~\eqref{eq:special_cart} uses fewer variables, it introduces more constraints. The results therefore suggest that the reduction in variables does not offset the additional constraint burden in this setting.

\begin{table}[htbp]
\centering
\caption{Computational comparison for ER-DD-SAA with CART in two-stage facility location problem.}
\label{tab:cart_comparison_facility}
\resizebox{\textwidth}{!}{%
\begin{tabular}{c|rrr|rrr|rrr}
\hline
\multirow{2}{*}{$N$} & \multicolumn{3}{c|}{Model \eqref{eq:er_dd_saa_cart}}
& \multicolumn{3}{c|}{Model \eqref{eq:special_cart}}
& \multicolumn{3}{c}{Gurobi ML} \\
\cline{2-4} \cline{5-7} \cline{8-10}
& \multicolumn{1}{c}{IS Cost} & \multicolumn{1}{c}{OOS Cost} & \multicolumn{1}{c|}{Time (s)}
& \multicolumn{1}{c}{IS Cost} & \multicolumn{1}{c}{OOS Cost} & \multicolumn{1}{c|}{Time (s)}
& \multicolumn{1}{c}{IS Cost} & \multicolumn{1}{c}{OOS Cost} & \multicolumn{1}{c}{Time (s)} \\ \hline
1000 & -2,003,034.42 & -1,853,490.67 & \textbf{49.21}  & -2,003,034.42 & -1,853,490.67 & 82.36    & -2,003,034.42 & -1,853,490.67 & 207.30   \\
2000 & -2,001,948.46 & -1,855,618.67 & \textbf{90.62}  & -2,001,948.46 & -1,855,618.67 & 758.63   & -2,001,948.46 & -1,855,618.67 & 898.03   \\
3000 & -2,000,245.25 & -1,856,664.56 & \textbf{94.64}  & -2,000,245.25 & -1,856,664.56 & 1,538.98 & -2,000,245.25 & -1,856,664.56 & 1,491.95 \\
4000 & -1,998,478.01 & -1,855,570.57 & \textbf{136.69} & -1,998,478.01 & -1,855,570.57 & 2,455.50 & -1,998,478.01 & -1,855,570.57 & 3,231.30 \\ \hline
\end{tabular}%
}
\end{table}






\end{document}